\documentclass[10pt,reqno]{amsart}
\usepackage{latexsym,amsmath,amssymb,amscd}
\usepackage[all]{xy}
\usepackage{enumerate}

\makeatletter
  \newcommand\@dotsep{4.5}
  \def\@tocline#1#2#3#4#5#6#7{\relax
     \ifnum #1>\c@tocdepth 
     \else
     \par \addpenalty\@secpenalty\addvspace{#2}%
     \begingroup \hyphenpenalty\@M
     \@ifempty{#4}{%
     \@tempdima\csname r@tocindent\number#1\endcsname\relax
        }{%
         \@tempdima#4\relax
           }%
      \parindent\z@ \leftskip#3\relax \advance\leftskip\@tempdima\relax
      \rightskip\@pnumwidth plus1em \parfillskip-\@pnumwidth
       #5\leavevmode\hskip-\@tempdima #6\relax
       \leaders\hbox{$\m@th
       \mkern \@dotsep mu\hbox{.}\mkern \@dotsep mu$}\hfill
       \hbox to\@pnumwidth{\@tocpagenum{#7}}\par
       \nobreak
        \endgroup
         \fi}
\makeatother 

\begin{document}

\makeatletter
\@addtoreset{figure}{section}
\def\thefigure{\thesection.\@arabic\c@figure}
\def\fps@figure{h,t}
\@addtoreset{table}{bsection}

\def\thetable{\thesection.\@arabic\c@table}
\def\fps@table{h, t}
\@addtoreset{equation}{section}
\def\theequation{
\arabic{equation}}
\makeatother

\newcommand{\bfi}{\bfseries\itshape}

\newtheorem{theorem}{Theorem}
\newtheorem{acknowledgment}[theorem]{Acknowledgment}
\newtheorem{algorithm}[theorem]{Algorithm}
\newtheorem{axiom}[theorem]{Axiom}
\newtheorem{case}[theorem]{Case}
\newtheorem{claim}[theorem]{Claim}
\newtheorem{conclusion}[theorem]{Conclusion}
\newtheorem{condition}[theorem]{Condition}
\newtheorem{conjecture}[theorem]{Conjecture}
\newtheorem{construction}[theorem]{Construction}
\newtheorem{corollary}[theorem]{Corollary}
\newtheorem{criterion}[theorem]{Criterion}
\newtheorem{data}[theorem]{Data}
\newtheorem{definition}[theorem]{Definition}
\newtheorem{example}[theorem]{Example}
\newtheorem{lemma}[theorem]{Lemma}
\newtheorem{notation}[theorem]{Notation}
\newtheorem{problem}[theorem]{Problem}
\newtheorem{proposition}[theorem]{Proposition}
\newtheorem{question}[theorem]{Question}
\newtheorem{remark}[theorem]{Remark}
\newtheorem{setting}[theorem]{Setting}
\numberwithin{theorem}{section}
\numberwithin{equation}{section}

\newcommand{\todo}[1]{\vspace{5 mm}\par \noindent
\framebox{\begin{minipage}[c]{0.85 \textwidth}
\tt #1 \end{minipage}}\vspace{5 mm}\par}

\renewcommand{\1}{{\bf 1}}

\newcommand{\hotimes}{\widehat\otimes}

\newcommand{\Ad}{{\rm Ad}}
\newcommand{\Alt}{{\rm Alt}\,}
\newcommand{\Ci}{{\mathcal C}^\infty}
\newcommand{\comp}{\circ}
\newcommand{\D}{\text{\bf D}}
\newcommand{\de}{{\rm d}}
\newcommand{\ev}{{\rm ev}}
\newcommand{\fimes}{\mathop{\times}\limits}
\newcommand{\id}{{\rm id}}
\newcommand{\ie}{{\rm i}}
\newcommand{\End}{{\rm End}\,}
\newcommand{\Fl}{{\rm Fl}}
\newcommand{\Gr}{{\rm Gr}}
\newcommand{\GL}{{\rm GL}}
\newcommand{\Hilb}{{\bf Hilb}\,}
\newcommand{\Hom}{{\rm Hom}\,}
\newcommand{\Ker}{{\rm Ker}\,}
\newcommand{\Kern}{\textbf{Kern}}
\newcommand{\Lie}{\textbf{L}}
\newcommand{\lf}{{\rm l}}
\newcommand{\pr}{{\rm pr}}
\newcommand{\R}{\mathbb R}
\newcommand{\Ran}{{\rm Ran}\,}
\newcommand{\RK}{{\mathcal P}{\mathcal K}^{-*}}
\newcommand{\spann}{{\rm span}}
\newcommand{\Tr}{{\rm Tr}\,}

\newcommand{\G}{{\rm G}}
\newcommand{\U}{{\rm U}}
\newcommand{\VB}{{\rm VB}}

\newcommand{\CC}{{\mathbb C}}
\newcommand{\RR}{{\mathbb R}}

\newcommand{\Ac}{{\mathcal A}}
\newcommand{\Bc}{{\mathcal B}}
\newcommand{\Cc}{{\mathcal C}}
\newcommand{\Dc}{{\mathcal D}}
\newcommand{\Ec}{{\mathcal E}}
\newcommand{\Fc}{{\mathcal F}}
\newcommand{\Gc}{{\mathcal G}}
\newcommand{\Hc}{{\mathcal H}}
\newcommand{\Kc}{{\mathcal K}}
\newcommand{\Lc}{{\mathcal L}}
\newcommand{\Oc}{{\mathcal O}}
\newcommand{\Pc}{{\mathcal P}}
\newcommand{\Qc}{{\mathcal Q}}
\newcommand{\Sc}{{\mathcal S}}
\newcommand{\Tc}{{\mathcal T}}
\newcommand{\Vc}{{\mathcal V}}
\newcommand{\Xc}{{\mathcal X}}
\newcommand{\Yc}{{\mathcal Y}}
\newcommand{\Zc}{{\mathcal Z}}

\renewcommand{\gg}{{\mathfrak g}}
\newcommand{\Gg}{{\mathfrak g}}
\newcommand{\hg}{{\mathfrak h}}
\newcommand{\mg}{{\mathfrak m}}
\newcommand{\nng}{{\mathfrak n}}
\newcommand{\Sg}{{\mathfrak S}}
\newcommand{\ug}{{\mathfrak u}}

\pagestyle{myheadings}
\markboth{}{}


\makeatletter
\title[Infinitesimal aspects of idempotents in Banach algebras]{Infinitesimal aspects of idempotents \\ in Banach algebras}
\author{Daniel Belti\c t\u a and  Jos\'e E. Gal\'e}
\address{Institute of Mathematics ``Simion
Stoilow'' of the Romanian Academy, 
RO-014700 Bucharest, Romania}
\email{Daniel.Beltita@imar.ro}
\address{Departamento de matem\'aticas, Facultad de Ciencias, 
Universidad de Zaragoza, 50009 Zaragoza, Spain}
\email{gale@unizar.es}
\thanks{
This research was partly  supported by Project MTM2013-42105-P, Fondos FEDER, Spain. 
The first-named author has also been supported by a Grant of the Romanian National Authority for Scientific Research, CNCS-UEFISCDI, project number PN-II-ID-PCE-2011-3-0131. 
The second-named author has also been supported by Project E-64, D.G. Arag\'on, Spain.}
\dedicatory{To Jean Esterle}
\date{4 November 2016}
\makeatother

\begin{abstract} 
We investigate infinitesimal properties of sets of ordered $n$-uples of  idempotents in a symmetric Banach $*$-algebra. These sets are called flag manifolds and carry several interesting bundles that hold an important role 
in some areas of operator theory. In this direction, we introduce and study Stiefel bundles on flag manifolds, which are extensions of the well known Stiefel bundles on Grassmannians. 
The main ingredient of our investigation is the notion of connection  on an infinite-dimensional bundle, and we survey some equivalent ocurrences of such a notion in the literature.  
\end{abstract}

\maketitle

\tableofcontents

\section{Introduction}

This paper is written in particular as a way to express our recognition of the important achievements of Jean Esterle in Mathematics. 
Among the great variety of research topics approached by him,  
there is that one of the geometry/topology associated with polynomial path connection in sets of idempotents 
in Banach algebras, 
see \cite{Es83}, \cite{EG04}, \cite{Gi03}, \cite{Tr85}. 
We wish to contribute to this proceedings volume in honor of Jean Esterle by pointing out several differential-geometric features that flag manifolds enjoy. 
These manifolds 
are sets of ordered $n$-uples of  idempotents in Banach algebras, 
naturally arising in problems related with differential structures 
of the functional analysis, operator theory, and particularly in spectral theory;
see for instance \cite{PR87}, \cite{MR92}, \cite{CPR90}, \cite{MS97}, \cite{CG99}, \cite{ACS01}, and the references therein. 

Flag manifolds carry several interesting bundles that hold an important role 
in some areas of operator theory like the theory of reproducing kernels, the noncommutative spectral theory, or the Cowen-Douglas theory. In this direction, 
we introduce here the notion of Stiefel bundle on a flag manifold and study some of its basic properties. Stiefel bundles on flag manifolds generalize the corresponding notion defined over the Grassmanian manifold associated with a single idempotent (see \cite{DG01}, \cite{DG02}). 
We investigate these notions using as a main tool 
the notion of connection on a fiber bundle, which is discussed in detail in Section~\ref{Sect2}. 

The theory of connections is of great importance 
in finite-dimensional differential and algebraic geometry, 
as well as in their applications to operator theory, 
see for instance \cite{CD78}.  
As regards the increasingly important geometric infinite-dimensional setting, 
there are several different ---though equivalent--- methods to introduce connections on Banach and more general bundles. 
This circle of ideas is 
certainly known to many experts in finite-dimensional differential geometry,  
but here we discuss it in a self-contained way with a view to some problems in operator theory, 
where one must work with infinite-dimensional manifolds modeled on Banach spaces. 

Our present paper is organized as follows. 
In Section~\ref{Sect2} we survey 
a number of different definitions 
of a connection and show how these definitions are related to each other. 
We choose to make an exposition in terms of exact sequences of fiber bundles. 
We also give some original results, particularly related to the pull-back operation. 
Section~\ref{Sect3} of this paper is devoted to introducing Stiefel bundles on flag manifolds associated to symmetric Banach $*$-algebras.  
After that, using some of the general methods developed in Section~\ref{Sect2}, 
we analyze some basic properties of Stiefel bundles, including transitive group actions, principal and linear connections on various bundles on flag manifolds, and related topics. 

\subsection*{General conventions and notation} 
Throughout this paper, by Banach space we mean a real or complex complete normable vector space $E$, 
so we do not assume that any norm on $E$ has been fixed. 
In this case we denote by $\Bc(E)$ its corresponding space of continuous linear operators. 
Unless otherwise mentioned, by manifold or Banach manifold we mean a smooth manifold modeled on a Banach space.  
Smoothness always refers here to Fr\'echet differentiability on Banach spaces. 
A Banach-Lie group is a Banach manifold $G$ endowed with a group operation $(x,y)\mapsto xy$, 
that is smooth and for which  the inversion $x\mapsto x^{-1}$ is also smooth. 
The Lie algebra of $G$ is $\gg:=T_{\1}G$, the tangent space at the unit element $\1\in G$. 
This is a Banach space endowed with a Lie bracket $[\cdot,\cdot]\colon\gg\times\gg\to\gg$ 
that is in particular a skew-symmetric bilinear map  
defined as $\frac{\partial}{\partial x}\big\vert_{x=\1}\frac{\partial}{\partial y}\big\vert_{y=\1}xyx^{-1}$. 
Thus, if $G=\Ac^\times$, the group of invertible elements in a unital Banach algebra $\Ac$, 
then $\gg=\Ac$ as Banach spaces, and the Lie bracket on $\Ac$ is the usual commutator $[a,b]=ab-ba$ for all $a,b\in\Ac$. 
Basic references for Lie groups, global analysis and differential geometry in infinite dimensions are 
\cite{Lr11}, \cite{La01}, \cite{KM97a}, and \cite{Up85}, 
but most of the relevant definitions are collected in Section~\ref{Sect2}, 
for the reader's convenience.

\section{Connections on principal and vector bundles}\label{Sect2}

\subsection{Connections on fiber bundles}

In this paper, by fiber bundle with total space $M$ and base space $Z$ we mean any smooth map $\varphi\colon M\to Z$ 
where $M$ and $Z$ are Banach manifolds. 

We are interested in two types of fiber bundles, the principal bundles and the vector bundles.

\begin{definition}\label{pb1}
\normalfont
Let $P$ and $Z$ be Banach manifolds and $G$ be a Banach-Lie group with a smooth right action 
$\mu\colon P\times G\to P$, $(x,g)\mapsto \mu_g(x)=\mu_x(g)$. 

A fiber bundle $\pi \colon P\to Z$ is a \emph{principal bundle} with structure group $G$, 
or \emph{principal $G$-bundle} for short, if the following conditions are satisfied:
\begin{enumerate}
\item\label{pb1_item1} 
The action $\mu$ is free, that is, 
for each $x\in P$ the mapping $G\to P$, $g\mapsto \mu_g(x)$ is injective. 
\item\label{pb1_item2} 
The orbits of $\mu$ are the fibers of $\pi $, hence 
$$
(\forall z\in Z)(\forall x_0\in\pi^{-1}(z))\quad 
\pi^{-1}(z)=\{\mu_g(x_0)\mid g\in G\}.
$$
\item\label{pb1_item3} 
The mapping $\pi$ is a locally trivial bundle with the fiber $G$. 
In other words, every $z\in Z$ has an open neighborhood $V\subseteq Z$ 
such that there exists a $G$-equivariant diffeomorphism 
$\psi\colon \pi ^{-1}(V)\to V\times G$ with the property that the diagram 
\begin{equation}\label{pb1_eq1}
\begin{CD} 
\pi ^{-1}(V) @>{\psi}>> V\times G \\
@V{\pi }VV @VVV \\
V @>{\id}>> V
\end{CD} 
\end{equation}
is commutative. 
\end{enumerate}
If $G$, $P$, and $Z$ are complex manifolds, the mappings  
$\mu$ and $\pi $ are holomorphic, and at every point $z\in Z$ the local trivialization $\psi$ can be chosen holomorphic, 
then we say that $\pi \colon P\to Z$ is a \emph{holomorphic principal $G$-bundle}. 
\end{definition}

\begin{remark}\label{pb2}
\normalfont
Condition~\eqref{pb1_item3} in Definition~\ref{pb1} shows that 
the fibers of $\pi $ (or, equivalently, the orbits of $\mu$) are submanifolds of~$P$. 
We also note that the commutativity of the diagram~\eqref{pb1_eq1} is equivalent to the fact 
that there exists a smooth mapping $\varphi\colon \pi ^{-1}(V)\to G$ such that 
$\phi(u)=(\pi(u),\varphi(u))$ and 
$\varphi(\mu_g(u))=\varphi(u)g$ for every $g\in G$ and $u\in \pi ^{-1}(V)$. 
\end{remark}

We say that a fiber bundle $\varphi\colon M\to Z$  is a (real or complex) \emph{vector bundle} if for every $z\in Z$ 
the fiber $\varphi^{-1}(z)$ has the structure of a (real or complex) Banach space and there exist 
a family of open subsets $V_j\subseteq Z$, (real or complex) Banach spaces $E_j$ and diffeomorphisms 
$\Psi_j\colon \varphi^{-1}(V_j)\to V_j\times E_j$ for all $j\in J$ with the following properties: 
First, for every $j\in J$ and $z\in V_j$ the mapping $x\mapsto \Psi_j^{-1}(z,x)$ is a linear topological isomorphism 
of Banach spaces from $E$ onto $\varphi^{-1}(z)$. 
Second, for all $j,k\in J$ with $V_j\cap V_k\ne\emptyset$ there exists a smooth map $\chi_{j,k}\colon V_j\cap V_k\to \Bc(E_k,E_j)$ 
with $(\Psi_j\circ\Psi_k^{-1})(z,x)=(z,\chi_{jk}(z)x)$ for all $x\in E_k$ and $z\in V_j\cap V_k$. 
The above data are called a system of local trivializations of the vector bundle $\varphi$. 
Besides the trivial vector bundles $V\times E\to V$, $(z,x)\mapsto z$, 
basic examples of vector bundles are the tangent bundles of smooth manifolds.

To any 
fiber bundle $\varphi\colon M\to Z$, 
one can naturally associate 
two well known vector bundle structures of the tangent space $TM$:
\begin{itemize}
\item[$\bullet$] $\tau_M\colon TM\to M$, the tangent bundle of the total space $M$.
\item[$\bullet$] $T\varphi\colon TM\to TZ$, the tangent map of $\varphi$.
\end{itemize}

We take as starting point a definition of connection on a fiber bundle which provides us with a suitable horizontal distribution.  

\begin{definition}\label{def11}
\normalfont
A {\it connection} on the bundle $\varphi\colon M\to Z$ is any smooth mapping
$\Phi\colon TM\to TM$ with the following properties:
\begin{itemize}
\item[{\rm(i)}] $\Phi\circ\Phi=\Phi$;
\item[{\rm(ii)}] the pair $(\Phi,\id_M)$ is an endomorphism of the bundle
$\tau_M\colon TM\to M$;
\item[{\rm(iii)}] for every $x\in M$, if we denote $\Phi_x:=\Phi|_{T_xM}\colon
T_xM\to T_xM$, then we have $\Ran(\Phi_x)=\Ker(T_x\varphi)$,
so that we get an exact sequence
$$
0\rightarrow H_xM\hookrightarrow T_xM\mathop{\longrightarrow}\limits^{\Phi_x}
T_xM
\mathop{\longrightarrow}\limits^{T_x\varphi} T_{\varphi(x)}Z
\rightarrow 0.
$$
\end{itemize}
Here $H_xM:=\Ker(\Phi_x)$ is a closed linear subspace of $T_xM$ called the space of {\it horizontal vectors} at
$x\in M$. Similarly, the space of {\it vertical vectors} at $x\in M$ is 
$\Vc_xM:=\Ker(T_x\varphi)=\Ran(\Phi_x)$. Then we have the direct sum decomposition $T_xM=H_xM\oplus\Vc_xM$, for every $x\in M$ (cf. \cite[subsect. 37.2]{KM97a}). We denote by $\Vc P$ the {\it vertical} subbundle of $TM$ -which always exists- whose fibers at each $x\in M$ are $\Vc_xM$. 
The existence of the connection $\Phi$ defines the {\it horizontal distribution} given by the {\it horizontal} subbundle of $TM$ with fibers $H_x M$, $x\in M$.
\end{definition}

In the special cases of principal or vector bundles one defines:
\begin{enumerate}
\item If $\pi\colon P\to Z$ is a principal $G$-bundle then a connection $\Phi$ on $\pi\colon P\to Z$ is called {\it principal} whenever it is $G$-equivariant, that is,
$$
T( \mu_g)\circ\Phi=\Phi\circ T( \mu_g)
$$
for all $g\in G$ (cf. \cite[subsect. 37.19]{KM97a}).

\item
If $\Pi\colon D\to Z$ is a vector bundle then a connection $\Phi$ on 
$\Pi\colon D\to Z$ is called {\it linear} if the pair $(\Phi, \id_{TZ})$ is an endomorphism of the vector bundle
$T\Pi\colon TD\to TZ$ (i.e., if $\Phi$ is linear on the fibers of the bundle $T\Pi$); see \cite[subsect. 37.27]{KM97a}.
\end{enumerate}

There are different but equivalent notions of connections for principal or vector infinite-dimensional bundles, which not always are well explained or explicitly indicated in the literature. Next, we show such equivalences in a unified way, through exact sequences. Our exposition relies on references \cite{El67}, \cite{KM97a}, 
\cite{La01}, \cite{Pe69}, \cite{Va74}, among others.

\begin{notation}\label{prelim}
\normalfont
We will use 
the following notation and terminology:
\begin{itemize}
\item[$\bullet$] For any maps $\alpha\colon A\to R$ and 
$\beta\colon B\to R$ 
we define their \textit{fiber product} by 
$$
A\fimes_{\alpha, R,\beta}B:=\{(a,b)\in A\times B\mid\alpha(a)=\beta(b)\},
$$
and we define the \textit{pull-back of $\beta$ by $\alpha$} as the map 
$$
\alpha^*(\beta)\colon A\fimes_{\alpha,R,\beta}B\to A, \quad 
(a,b)\mapsto a.
$$
If no confusion is possible as regards the maps $\alpha$ and $\beta$, 
then we denote simply 
$$
A\fimes_{\alpha, R,\beta}B=A\fimes_{ R} B.
$$ 
See \cite[Ch. II, \S 2]{La01} 
for a discussion of fiber products of smooth maps. 
\item[$\bullet$] If $\alpha\colon E\to A$ and $\beta\colon F\to B$ 
are vector bundles and $f\colon E\to F$ and $\tilde{f}\colon A\to B$ 
are smooth maps such that the pair $(f,\tilde{f})$ 
is a morphism of vector bundles, then we define 
$$
f!\colon E\to A\fimes_B F,\quad \xi\mapsto(\alpha(\xi),f(\xi)).
$$
So the pair $(f!,\id_A)$ is a morphism of vector bundles 
from $\alpha\colon E\to A$ to $\tilde{f}^*(\beta)\colon A\fimes_B F\to A$.
\item[$\bullet$] If $G$ is a Banach-Lie group and $\alpha\colon E\to A$ 
is a vector bundle endowed with a smooth action from the right 
$\mu\colon E\times G\to E$ that maps every fiber of $E$ onto itself 
and is fiberwise linear, then we say that $\alpha$ is a 
\textit{vector $G$-bundle}. 
A \textit{morphism of vector $G$-bundles} is a morphism of vector bundles  
which is $G$-equivariant. 
\end{itemize}
\end{notation}

\subsection{Principal connections}

Let $\pi\colon P\to Z$ be a principal bundle 
with the structure Banach-Lie group $G$ and free action 
$\mu\colon P\times G\to P$. 
The induced action of $G$ on the tangent manifold $TP$ is defined by the tangent map $T\mu$, given by 
$$
TP\times G\to TP, \quad (v,g)\mapsto (T\mu_g)v.
$$

Let $\Gg$ be the Lie algebra of $G$. One considers the right action of $G$ on the manifold $P\times\Gg$ defined by 
$$
(P\times\Gg)\times G\to P\times\Gg,\quad 
((u,X),g)\mapsto(\mu(u,g),\Ad_G(g^{-1})X).
$$
 
Denote by $\iota\colon P\times\Gg\to \Vc P\subset TP$ the vector bundle isomorphism defined by the infinitesimal action of the Lie algebra $\Gg$ of $G$; that is, 
$$
(u,X)\mathop{\mapsto}\limits^{\iota}(T_{\1}\mu(u,\cdot))X
=
T_{u,\1}\mu(0_u,X)\in\Vc_uP\subset T_uP
$$
(\cite[Th. 37.18(2)]{KM97a}). The vector field $u\mapsto\iota(u,X)$ is called the {\it fundamental vector field} corresponding to $X\in\Gg$.

Notate for a moment
$u\cdot g=\mu(u,g)$ for $u\in P$, $g\in G$. Then we have 
$$
\iota(u,X)\cdot g=\left((T\mu_u)X\right)\cdot g=T\mu_gT\mu_uX
$$
and
$$
\begin{aligned}
\iota\left((u,X)\cdot g\right)
&=\iota\left((u\cdot g, \Ad_G(g^{-1})X)\right)\\
&=T\mu_{u\cdot g}\Ad_G(g^{-1})X
=T\Theta_{g,u\cdot g}(X)
=T\mu_gT\mu_uX
\end{aligned}
$$
where for every $h\in G$, 
$$
\Theta_{g,u\cdot g}(h)=\mu_{u\cdot g}(g^{-1}hg)=(u\cdot h)\cdot g
=\left(\mu_g\circ\mu_u\right)(h).
$$ 

The above means that  $\iota$ is $G$-equivariant, and hence $\iota$ is a morphism  of vector $G$-bundles between the trivial 
fiber bundle $\pr_1^{P\times\Gg}\colon P\times\Gg\to P$, $(u,X)\mapsto u$,  
and $\tau_P\colon TP\to P$. Moreover, the map 
$P\times\Gg\buildrel{\iota}\over\longrightarrow TP$ is injective since the $G$ acts freely on $P$, and the map $TP\buildrel{(T\pi)!}\over\longrightarrow P\fimes_{\pi,Z,\tau_Z}TZ$, 
$(T\pi)!(v):=((\tau_{P})(v), T\pi(v))$, $\forall v\in TP$, is clearly  surjective. Then it follows that the horizontal arrows $\iota$ and $(T\pi)!$ of the commutative diagram 
\begin{equation}\label{D1}
\begin{CD}
P\times\Gg @>{\iota}>> TP @>{(T\pi)!}>> P\fimes_Z TZ \\
@VV{\pr_1^{P\times\Gg}}V @VV{\tau_P}V @VV{\pi^*(\tau_Z)}V \\
P @>{\id_P}>> P @>{\id_P}>> P
\end{CD}
\end{equation} 
are $G$-equivariant, 
and the upper row of this diagram defines 
an exact sequence of vector $G$-bundles.  

In the following theorem, we summarize different approaches to the notion of connection which are found in the literature. Put 
$q:=\pr_2^{P\times\Gg}\circ \iota^{-1}\colon\Vc P\to\Gg$.

\begin{theorem}\label{principalCARACT}
Let $\pi\colon P\to Z$ be a principal Banach bundle 
with structure Banach-Lie group $G$ and free action $\mu\colon P\times G\to P$. The following assertions are equivalent:
\begin{itemize}
\item[\rm{(1)}] There exists a fiberwise linear and smooth $G$-equivariant map
$$
\chi\colon P\fimes_{\pi,Z,\tau_Z}TZ
\to TP\ \hbox{ such that }\ (T\pi)!\circ\chi=\id_{\Pc\fimes_{\pi,Z,\tau_Z}TZ}.
$$
\item[\rm{(2)}]  There exists a morphism of vector $G$-bundles
$\psi\colon TP\to P\times\Gg$ such that $\psi\circ\iota=\id_{P\times\Gg}$.
\item[\rm{(3)}]  There exists a smooth map $\omega\colon TP\to\Gg$ such that
\begin{equation}\label{W1}
(\forall X\in\Gg)(\forall u\in P)\quad
\omega(T_{\1}(\mu(u,\cdot))X)=X
\end{equation}
and
\begin{equation}\label{W2}
(\forall g\in G)\quad \omega\circ T(\mu(\cdot,g))=\Ad_G(g^{-1})\circ\omega.
\end{equation}
\item[\rm{(4)}]  There exists a principal connection $\Phi$ on $\pi\colon P\to Z$; that is, $\Phi\colon TP\to\Vc P$ is a connection such that 
$$
T(\mu_g)\circ\Phi=\Phi\circ T(\mu_g),\quad g\in G.
$$
\end{itemize}

Then, if any of the above points (1)-(4) holds, the relationship between $\chi$, $\psi$, $\omega$, $\Phi$ is given by
\begin{equation}\label{D3}
\iota\circ\psi+\chi\circ(T\pi)!=\id_{TP},
\end{equation}
\begin{equation}\label{D4}
(\forall v\in TP)\quad \psi(v)=(\tau_P(v),\omega(v))\in P\times\Gg,
\end{equation}
\begin{equation}\label{D5}
\Phi=\iota\circ\psi.
\end{equation}
\begin{equation}\label{D5bis}
\omega=q\circ\Phi.
\end{equation}
\end{theorem}

\begin{proof}
The condition stated for $\chi$ in (1) of the statement means that 
$\chi$ splits the short exact sequence of vector $G$-bundles
defined by \eqref{D1}, 
that is, 
\begin{equation}\label{D2}
0\to\pr_1^{P\times\Gg} \mathop{\longrightarrow}\limits^{\iota} \tau_P
 \mathop{\longrightarrow}\limits^{(T\pi)!} \pi^*(\tau_Z) 
\to 0. 
\end{equation} 
In turn, the splitting of \eqref{D1} through $\chi$ is equivalent, in standard way, to the existence 
of a morphism of vector $G$-bundles
$\psi\colon TP\to P\times\Gg$ such that $\psi\circ\iota=\id_{P\times\Gg}$. 

In fact, the relationship between $\psi$ and $\chi$ is given by the equation
$$
\iota\circ\psi+\chi\circ(T\pi)!=\id_{TP}.
$$

The two terms in the left-hand side of this equation are 
idempotents endomorphisms of the vector $G$-bundle $\tau_P\colon TP\to P$ 
whose composition is equal to $0$. 
Since $(T\pi)!$ is surjective and $\iota$ is injective, 
we can use~(\ref{D3}) to compute each one of the maps 
$\psi$ and $\chi$ if we know the other one.

This gives us the equivalence between (1) and (2) of the statement.

On the other hand, there exists a smooth map 
$\omega\colon TP\to\Gg$ such that 
$$
(\forall v\in TP)\quad \psi(v)=(\tau_P(v),\omega(v))\in P\times\Gg,
$$
since the map $\psi\colon TP\to P\times\Gg$ is a bundle morphism over $P$. 
Then the condition $\psi\circ\iota=\id_{P\times\Gg}$ 
is equivalent to 
$$
(\forall X\in\Gg)(\forall u\in P)\quad 
\omega(T_{\1}(\mu_u)X)=X
$$
and the condition that $\psi$ should be $G$-equivariant is equivalent to 
$$
(\forall g\in G)\quad \omega\circ T(\mu_g)=\Ad_G(g^{-1})\circ\omega.
$$

Thus there exists a natural one-to-one correspondence between 
the splittings of the exact sequence~\eqref{D2} 
of vector $G$-bundles over $P$ on the one hand, 
and maps $\omega$ satisfying (\ref{W1}) and (\ref{W2}) on the principal bundle $\pi\colon P\to Z$ 
on the other hand. 
(See also Theorem~1.1 in \cite{Va74}.) This gives us the equivalence between points (2) and (3) -or (1) and (3)- of the statement.

Set now $\Phi:=\iota\circ\psi$. As said before, $\Phi$ is an idempotent endomorphism of the vector $G$-bundle $\tau_P$. In particular one has  $\Ran(\Phi_u)=\Ker(\id_{TP}-\Phi_u)$ for every $u\in P$, and also, by (\ref{D3}), 
$\id_{TP}-\Phi=\chi\circ(T\pi)!$ where $\chi$ is injective. Thus we get 
$\Ran(\Phi_u)=\Ker(T_u\pi)$ for all $u\in P$. Moreover, the equality 
$T(\mu_g)\circ\Phi=\Phi\circ T(\mu_g)$, for all $g\in G$, follows readily by the $G$-equivariance of $\iota$ and $\psi$.

Conversely, for a given $\Phi\colon TP\to\Vc P$ as in (4), set 
$\omega(v):=T_\1(\mu_{\tau_P(v)})^{-1}\Phi(v)$, $v\in TP$ (recall that $\iota$ is a trivialization of the vertical subbundle). Then $\omega$ is a smooth map satisfying (\ref{W1}) by definition and (\ref{W2}) by the $G$-equivariance of $\Phi$. 

Finally, \eqref{D5bis} is easily deduced: 
$q\circ\Phi=\pr_2^{P\times\Gg}\circ\iota^{-1}\circ\psi=\pr_2^{P\times\Gg}\circ\psi=\omega$.

Thus we have proved the theorem.  
\end{proof}

As regards the bundle homomorphisms appearing in the above theorem, the mapping $\chi$ clearly reflects the idea to lift the tangent space $TZ$ of the base space $Z$ as the space of horizontal vectors in $TP$. This map is introduced as 
{\it connection} in \cite{Pe69}, for instance. The mapping $\omega$ corresponds to the classical notion of {\it connection form} associated with a given principal connection (cf. \cite[p. 387]{KM97a}; see also \cite{KN63}), and the mapping $\psi$ can be seen as a kind of trivialization of the vertical bundle $\Vc P\to P$.

\subsection{Connections for vector bundles}

Let $\Pi\colon D\to Z$ be a vector bundle.
Let $\Vc D=\Ker(T\Pi)$ ($\subseteq TD$) be the vertical part of the tangent bundle $\tau_D\colon TD\to D$. 
A useful description of $\Vc D$ can be obtained by considering the fibered product 
$D\fimes_Z D:=\{(\xi_1,\xi_2)\in D\times D\mid\Pi(\xi_1)=\Pi(\xi_2)\}$ 
along with the natural maps $r_j\colon D\fimes_Z D\to D$, $r_j(\xi_1,\xi_2)=\xi_j$ for $j=1,2$. 
Define for every $(\xi_1,\xi_2)\in D\fimes_Z D$ the path 
$c_{\xi_1,\xi_2}\colon\R\to D$, $c_{\xi_1,\xi_2}(t)=\xi_1+t\xi_2$ that passes through $\xi_1$ in the direction $\xi_2$. 
Then it is easily seen that we have a well-defined diffeomorphism 
\begin{equation}\label{Epsilon}
\varepsilon\colon D\fimes_Z D\to \Vc D,\quad \varepsilon(\xi_1,\xi_2)=\dot c_{\xi_1,\xi_2}(0)\in T_{\xi_1} D,
\end{equation} 
which is in fact an isomorphism between the vector bundles 
$r_1\colon D\fimes_Z D\to D$ and $\tau_D\vert_{\Vc D}\colon \Vc D\to D$. 
We then get a natural mapping 
$r:=r_2\circ\varepsilon^{-1} \colon\Vc D\to D$ and the pair 
$(r,\Pi)$ is a homomorphism of vector bundles from $\tau_D\vert_{\Vc D}\colon \Vc D\to D$ 
to $\Pi\colon D\to Z$. The manifold $D\fimes_Z D$ and morphism $\varepsilon$ play, for vector bundles, the same role as the manifold $P\times\Gg$ and morphism $\iota$ play for principal bundles. 

In the present setting, we have the exact sequences of vector bundles given by 
\begin{equation}\label{ESvb}
\begin{CD}
D\fimes_Z D @>\varepsilon>> TD @>{T\pi!}>> TZ\fimes_Z D\\
@VV{\pr_1}V @VV{\tau_D}V @VV{\Pi^*(\tau_Z)}V \\
D @>{\id_{D}}>> D @>{\id_{D}}>> D.
\end{CD}
\end{equation}
Similarly to the case of principal bundles we next show different but equivalent presentations of connections on Banach vector bundlles. Recall, 
$r=r_2\circ\varepsilon^{-1}$.

\begin{theorem}\label{caractVECTOR}
For any vector bundle $\Pi\colon D\to Z$, the following assertions are equivalent.
\begin{itemize}
\item[\rm{(1)}] There exists a smooth map
$
\gamma\colon TZ\fimes_{\tau_Z,Z,\Pi}D\to TD
$
such that: 
\begin{itemize}
\item[$1^\circ$] For all $(v,\xi)\in TZ\fimes_Z D$ 
we have $T\Pi(\gamma(v,\xi))=v$ and $\tau_D(\gamma(v,\xi))=\xi$. 
\item[$2^\circ$] For all $z\in Z$ and $\xi\in D_z$ 
the mapping $\gamma(\cdot,\xi)\colon T_zZ\to T_\xi D$ is linear. 
\item[$3^\circ$] For all $z\in Z$ and $v\in T_zZ$ 
the mapping $\gamma(v,\cdot)\colon D_z\to(TD)_v$ is linear, 
where $(TD)_v$ stands for the fiber of the vector bundle 
$T\Pi\colon TD\to TZ$ over $v\in TZ$. 
\end{itemize}
\item[\rm{(2)}] There exists a smooth map 
$\lambda\colon TD\to D\fimes_{\Pi,Z,\Pi}D$ 
with $
\lambda\circ\varepsilon=\id_{D\fimes_{\Pi,Z,\Pi}D},
$ 
so that the pair $(\lambda,\id_D)$ is a morphism of vector bundles from 
$\tau_D\colon TD\to D$ to $\pr_1\colon D\times_Z D\to D$.
\item[\rm{(3)}] There exists a smooth mapping $K\colon TD\to D$ such that:
\begin{itemize}
\item[{\rm(i)}] $\Pi\circ K=\Pi\circ\tau_D$ and 
the pair  $(K,\Pi)$ defines
a morphism of vector bundles from $\tau_D$ to $\Pi$.
\item[{\rm(ii)}] For all $(\xi,\eta)\in D\fimes_{\Pi,Z,\Pi}D$ we have $K(\varepsilon(\xi,\eta))=\eta$.
\end{itemize}
\item[\rm{(4)}] One has a linear connection $\Phi\colon D\to Z$, in the sense of Definition~\ref{def11}(2).
\end{itemize}
The relationship between the above mappings is given by the formulas: 
\allowdisplaybreaks
\begin{gather}\label{D11}
\varepsilon\circ\lambda+\gamma\circ(T\Pi)!=\id_{TD};\\ 
\label{D13}
\lambda=(\tau_D,K)\  (\hbox{that is, } K=r_2\circ\lambda) \hbox{ on } TD; \\
\Phi:=\varepsilon\circ \lambda; \nonumber \\
K=r\circ\Phi.\nonumber
\end{gather}  
\end{theorem}

\begin{proof}

$(1)\Leftrightarrow(2)$. Conditions $1^\circ$--$3^\circ$ in (1) are equivalent 
to the fact that the map $\gamma\colon TZ\fimes_Z D\to TD$ defines 
splittings both for the exact sequence of vector bundles 
$$
\pr_1 D\buildrel{\varepsilon}\over\longrightarrow \tau_D\buildrel{(T\Pi)!}\over\longrightarrow \Pi*(\tau_Z)
$$
over $D$, that is (\ref{ESvb}), and the exact sequence of vector bundles 
\begin{equation}
\begin{CD}
TZ\fimes_Z D @>\gamma>> TD @>{\tau_D!}>> D\fimes_Z TZ\\
@VV{\pr_1}V @VV{T\Pi}V @VV{\tau_D^*(\Pi)}V \\
TZ @>{\id_{TZ}}>> TZ @>{\id_{TZ}}>> TZ
\end{CD}
\end{equation}
over $TZ$. 
In fact we have $(T\Pi)!(\xi)=(T\Pi(\xi),\tau_D(\xi))$ 
and $\tau_D!(\xi)=(\tau_D(\xi),T\Pi(\xi))$ for all $\xi\in TD$ 
(see \cite{Pe69}, the end of \cite[Ch.~IV, \S 3]{La01},
and \cite{Va74})

Since $\gamma$ splits the exact sequence of vector bundles~(\ref{ESvb}),
it follows that there exists a unique map
$\lambda\colon TD\to D\fimes_{\Pi,Z,\Pi}D$ such that
$$
\varepsilon\circ\lambda+\gamma\circ(T\Pi)!=\id_{TD},
$$
or, equivalently, it satisfies $\lambda\circ\varepsilon=\id_{D\fimes_{\Pi,Z,\Pi}D}$. As in the case of principal bundles, with $\chi$ and $\psi$, the maps $\gamma$ and $\lambda$ determines each other. 

$(2)\Leftrightarrow(3)$.
For $r_1,r_2$ as prior to the theorem, $\Pi\circ r_2=\Pi\circ r_1$ so the pair $(r_2,\Pi)$ is
a morphism of vector bundles from $\pr_1$ to $\Pi$, 
which is moreover a fiberwise isomorphism. 

Now assume that  
$\lambda\colon TD\to D\fimes_ZD$ is a smooth map such that the pair $(\lambda,\id_D)$ is a morphism of vector bundles 
from $\tau_D\colon TD\to D$ to $\pr_1\colon D\fimes_Z D\to D$. 
Then the map  
$$
K:=r_2\circ\lambda\colon TD\to D.
$$
satisfies the conditions of (3). 

Conversely, for a given map $K$ as in (3) one can define the smooth map 
$$
\lambda\colon TD\to D\fimes_Z D,\quad \lambda(v)=(\tau_E(v),K(v)).
$$
Thus the pair $(\lambda,\id_D)$ is a morphism of vector bundles 
from $\tau_D\colon TD\to D$ to $\pr_1\colon D\fimes_Z D\to D$. 
Moreover, $\lambda$ splits 
the exact sequence of vector bundles~(\ref{ESvb}). 
In fact, for every $(\xi,\eta)\in D\fimes_Z D$ we have 
$$
\lambda(\varepsilon(\xi,\eta))=(\tau_D(\varepsilon(\xi,\eta)),K(\varepsilon(\xi,\eta)))
=(\xi,\eta)
$$
according to condition~(i) in (3). Thus 
$\lambda\circ\varepsilon=\id_{D\fimes_Z D}$.

Let us note also that the above constructions are inverse to one another, 
so that we get a one-to-one correspondence between maps $\lambda$ and maps $K$; cf.~subsection~37.27 in \cite{KM97a}.

$(2)\Leftrightarrow(4)$. This is a matter of fact that the relations 
$$
\Phi:=\varepsilon\circ \lambda\ \hbox{ and } \lambda:=\varepsilon^{-1}\circ\Phi 
$$
establishes the link between the maps $\lambda$ of (2) and the maps $\Phi$ of (4). 

Finally, $r\circ\Phi=r_2\circ\varepsilon^{-1}\circ\varepsilon\circ\lambda=r_2\lambda=K$, and the proof is over.
\end{proof}

\begin{remark}\label{alternacoCON}
\normalfont
The map $\gamma$ given in Theorem \ref{caractVECTOR} (1) is taken as a 
{\it connection} in several papers or books (see for example \cite{Pe69}, 
\cite[Ch.~IV,\S 3]{La01}, and \cite{Va74}). The map $K$ is usually called 
{\it connection map}, or {\it connector}, see \cite[\S 37.27]{KM97a}. 
The above notion of connection map $K$ coincides with the one 
introduced in \cite[page 172]{El67}. 
To explain this, we shall use the notation of the preceding theorem. 
In addition, denote by $\textbf{Z}$ the model space of the manifold $Z$ and  
by $\textbf{E}$ the fiber of $\alpha\colon D\to Z$, 
and let $\varphi\colon U\to\textbf{Z}$ be a local chart of $Z$ 
and $\Psi\colon\Pi^{-1}(U)\to\varphi(U)\times\textbf{E}$ 
a local trivialization of the bundle $\Pi$. 
Then the local representative of a connection 
$\gamma\colon TZ\fimes_Z D\to TD$ looks like this: 
\begin{equation}\label{K12}
(T\Psi\circ\gamma\circ(T\varphi^{-1}\times\Psi^{-1}))((x,y),(x,\xi))
=(x,\xi,y,-\Gamma_\varphi(x)(y,\xi))
\end{equation}
whenever $x\in\varphi(U)\subseteq\textbf{Z}$, $y\in\textbf{Z}$, 
and $\xi\in\textbf{E}$, 
where $\Gamma_\varphi\colon\varphi(U)\to\Bc(\textbf{Z},\textbf{E};\textbf{E})$ 
is the \textit{Christoffel symbol} 
(cf. \cite[page A101]{Pe69}). 
Here we denote by $\Bc(\textbf{B},\textbf{E};\textbf{E})$ the space of continuous 
bilinear maps $\textbf{Z}\times\textbf{E}\to\textbf{E}$.
It then follows by the above local expression of $\gamma$ 
along with the definiton of the connection map $K\colon TE\to E$ that its local expression is 
\begin{equation}\label{K13}
(\Psi\circ K\circ T\Psi^{-1})(x,\xi,y,\eta)
=(x,\eta+\Gamma_\varphi(x)(y,\xi))
\end{equation}
whenever $x\in\varphi(U)\subseteq\textbf{Z}$, $y\in\textbf{Z}$, 
and $\xi,\eta\in\textbf{E}$. 
In fact, we get by (\ref{D11})~and~(\ref{D13}) that 
$$
(\forall v\in TD)\quad 
\varepsilon(\tau_D(v),K(v))=v-\gamma(T\Pi(v),\tau_D(v)).
$$
By expressing this equation in local coordinates we get 
$$
\begin{aligned}
(T\Psi\circ\varepsilon\circ(\Psi^{-1}\times\Psi^{-1}))(x,\xi,\Psi(K(v)))
=&(x,\xi,y,\eta)  \\
 &-(T\Psi\circ\gamma\circ(T\varphi^{-1}\times\Psi^{-1}))((x,y),(x,\xi)) \\
=&(x,\xi,y,\eta) -(x,\xi,y,-\Gamma_\varphi(y,\xi)) \\
=&(x,\xi,0,\eta+\Gamma_\varphi(y,\xi)).  
\end{aligned}
$$
Since $v=(T\Psi^{-1})(x,\xi,y,\eta)$, 
it now follows by \eqref{Epsilon} that 
$(\Psi\circ K\circ T\Psi^{-1})(x,\xi,y,\eta)
=(x,\eta+\Gamma_\varphi(x)(y,\xi))$, as claimed. 
And thus the definition of $K$ in Theorem \ref{caractVECTOR} is indeed equivalent to 
the Definition in \cite[page 172]{El67}. 
\end{remark}

\subsection{Linear connections induced from principal ones}\label{subs2.4}

It is well known that principal connections on principal $G$-bundles $\pi$ induce linear connections on vector bundles $\Pi$ associated with $\pi$ through representations of the structure group $G$. In the infinite dimensional setting, a very suitable induction procedure  
is given in \cite[\S 37.24]{KM97a} in terms of connections $\Phi$. Here, we 
show how such a procedure looks when one takes maps $\chi$ and $\gamma$ as references for connections.

So let $\pi\colon P\to Z$ be a principal Banach $G$-bundle with action 
$\mu\colon P\times G\to P$. 
Let $\rho\colon G\to\Bc(\textbf{V})$ is 
a smooth representation of $G$ by bounded linear operators on 
a Banach space $\textbf{V}$, and denote by 
$$
\Pi\colon D=P\times_G\textbf{V}\to Z,\quad [(u,x)]\mapsto\pi(u)$$
the \textit{associated vector bundle} 
(see \cite[subsect. 6.5]{Bo67} and \cite[subsect. 37.12]{KM97a}). 
Recall that $P\times_G\textbf{V}$ denotes the quotient of 
$P\times\textbf{V}$ with respect to the equivalence relation defined by 
$$
(\forall g\in G)\quad (u,x)\sim(\mu(u,g),\rho(g^{-1})x)=:\bar{\mu}(g)(u,x)$$
whenever $(u,x)\in P\times\textbf{V}$, 
and we denote by $[(u,x)]$ the equivalence class of any pair $(u,x)$. 
In this way, $\Pi\colon D\to Z$ is a vector $G$-bundle. 

Furthermore, the tangent manifold $TG$ is the semidirect product of groups 
$TG\equiv G \ltimes_{\Ad_{G}}\Gg$ defined by the adjoint action of $G$ on $\Gg$; 
see \cite[Cor. 38.10]{KM97a}. The
multiplication in the group $TG$ is given by 
$$
(g_1,X_1)(g_2,X_2)=(g_1 g_2,\Ad_{G}(g_2^{-1})X_1+X_2),
\quad (g_1,g_2\in G; X_1,X_2\in\Gg).
$$
Then the tangent bundle $T\pi\colon TP\to TZ$ is a principal bundle with the structure 
group $TG$ and right action $T\mu\colon T\Pc\times TG\to T\Pc$  
(\cite[Th. 37.18(1)]{KM97a}), and the tangent map of the representation $\rho$ can be viewed as the smooth representation
$T\rho\colon G\ltimes_{\Ad_{G}}\Gg\to\Bc(\textbf{V}\oplus\textbf{V})$ defined by
$$
(g,X)\mapsto 
\begin{pmatrix}
\rho(g)&0\cr 0 & \rho(g)\end{pmatrix}
\begin{pmatrix} 1& 0 \cr d\rho(X)  & 1\end{pmatrix}
=\begin{pmatrix} \rho(g)& 0 \cr \rho(g) d\rho(X) & \rho(g)
\end{pmatrix},
$$  
where the resulting matrix is to be understood as acting on vectors of $\textbf{V}\oplus\textbf{V}$ written in column form.
Using the representation $T\rho$, the tangent bundle  
of the vector bundle $\Pi\colon D=\Pc\times_{G}\textbf{E}\to Z$  
can be described as the vector bundle 
$$
\tau_D\colon TD=TP\times_{G\ltimes_{\Ad_{G}}\Gg}(\textbf{V}\dot+\textbf{V})
\to P\times_{G}\textbf{V}=D,
$$
which is associated to the principal bundle $T\pi\colon TP\to TZ$ and is defined by 
$$
\tau_D\colon [(v_u,(x,y))]\mapsto [(u,x)] \qquad (v_u\in T_u P; x,y\in\bf V)
$$ 
(\cite[Th. 37.18(4)]{KM97a}).

Let now $\chi\colon P\fimes_Z TZ\to TP$ be a connection 
on the principal bundle $\pi\colon P\to Z$. 
Denote by 
$$
q\colon P\times\textbf{V}\to P\times_G\textbf{V}=D,
\quad (u,x)\mapsto[(u,x)]
$$
the quotient map, and define 
$$
j\colon\textbf{V}\to\textbf{V}\times\textbf{V},\quad x\mapsto(x,0).
$$
Then the composition 
$$
(P\fimes_Z TZ)\times\textbf{V}
\mathop{\longrightarrow}^{\chi\times j}
TP\times(\textbf{V}\times\textbf{V})=T(P\times\textbf{V})
\mathop{\longrightarrow}^{Tq}TD
$$
has the property 
$$
\begin{aligned}
Tq((\chi\times j)((\mu(u,g),v),\rho(h^{-1})x))
 &=Tq(\chi(\mu(u,g),v),\rho(g^{-1})x,0) \\
 &=Tq((T\mu(\cdot,g)\circ\chi)(u,v),\rho(g^{-1})x,0) \\
 &=T(q\circ\bar{\mu}(g))(\chi(u,v),x,0) \\
 &=Tq(\chi(u,v),x,0) \\
 &=Tq((\chi\times j)((u,v),x))
\end{aligned}
$$
whenever $g\in G$, $x\in\textbf{V}$ and $(u,v)\in TZ$. 
(Note that the second of the above equalities follows by 
the $G$-equivariance property of $\chi$. 
We also note that a similar calculation 
can be found  in \cite[subsect. 37.24]{KM97a}.) 
Consequently the map 
$$
\gamma\colon TZ\fimes_Z D\to TD,
\quad (v,[(u,x)])\mapsto Tq((\chi\times j)((u,v),x))
$$
is well defined. 
Since the quotient map 
$$
(P\fimes_Z TZ)\times\textbf{V} \to 
TZ\fimes_Z(P\times_G\textbf{V}) =TZ\fimes_Z D, 
\quad ((u,v),x)\mapsto(v,[(u,x)])
$$
is a submersion and $Tq\circ(\chi\times j)$ is smooth, it follows by 
Corollary~4.8 in \cite{Up85} 
that $\gamma$ is smooth. 
Moreover, it is easy to see that conditions $1^\circ$--$3^\circ$ 
in Theorem~\ref{caractVECTOR}(3) are satisfied, 
hence $\gamma\colon TZ\fimes_Z D\to TD$ is actually a connection 
on the vector bundle $\Pi\colon D\to Z$, 
and it was constructed 
out of the connection $\chi\colon P\fimes_Z TZ\to TP$ 
on the principal bundle $\pi\colon P\to Z$. 

For the sake of completeness let us recall that every vector bundle 
is associated with its frame bundle, which is a principal bundle 
(see \cite{Bo67} and \cite[Sect. 2]{Va74}). 
In this special case the above construction can be found in \cite[Sect. 1]{Pe69}.

\subsection{Covariant derivatives}

Whereas a connection on a fiber bundle is introduced by algebraic tools, as a way to make differentation between different tangent spaces possible (through a suitable smooth family of horizontal tangent subspaces), the differential calculus in itself, on such a bundle, relies on the notion of the covariant derivative associated to the connection. We are not dealing with covariant derivatives in detail, but it sounds sensible to recall its definition and some properties.

Let $\Phi$, $K$ be a linear connection on a Banach vector bundle $\Pi$ and its corresponding connection map, respectively.  
Let $\Omega^1(Z,D)$ denote the space of locally defined 
smooth differential 1-forms on $Z$ with values in the bundle $\Pi\colon D\to Z$,  
and let $\Omega^0(Z,D)$ denote the space of locally defined smooth sections of the vector bundle~$\Pi$. 

The \textit{covariant derivative} for $\Phi$, or alternatively for $K$, 
is the linear mapping  
$\nabla\colon\Omega^0(Z,D)\to\Omega^1(Z,D)$, defined for every $\sigma\in\Omega^0(Z,D)$ 
by the composition 
$$
\nabla\sigma\colon  
TZ\mathop{\longrightarrow}\limits^{T\sigma} TD\mathop{\longrightarrow}\limits^{\Phi} 
\Vc D\mathop{\longrightarrow}\limits^{r}D
$$
that is, $\nabla\sigma=(r\circ\Phi)\circ T\sigma=K\circ T\sigma$. 

Covariant derivatives on Banach bundles enjoy many of the properties of covariant derivatives in finite dimensions. For instance, $\nabla$ can be locally expressed in terms of connection forms $\omega\in\Omega^1(P,\Gg)$ when the vector bundle is associated with a principal bundle (via a representation of the structure group). In this way covariant derivatives and their compatibility with Hermitian and complex structures, as well as the positivity of their curvature forms, on Banach vector bundles have been studied in \cite{BG14}, \cite{BG15a}, \cite{BG15b} with particular attention on the geometry of reproducing kernels.

\subsection{Pull-backs of connections}\label{sect2}

Pull-backs of connections on various types of finite-dimensional bundles
have been studied in several papers; see for instance
\cite{NR61}, \cite{NR63},
\cite{Le68},
\cite{Sch80}, 
\cite{PR86}.

In \cite[Prop. 6.6]{BG14}, a result on pull-back of connections $\Phi$ on Banach vector bundles is given, which is suitable for application in the study of differential geometric properties of reproducing kernels (see also \cite{BG15a}, \cite{BG15b}). 
Unlike most constructions of the pull-backs of connections in the literature, 
the method provided in \cite{BG14} is more direct in the sense that it requires 
neither the connection map, nor any connection forms, nor the covariant derivative, 
but rather the connection $\Phi$ itself. However, we wish here to show another type of pull-back operation relying on the connectors associated with connections, which also works in infinite dimensions, since connectors are tools frequently used in applications.

\begin{proposition}\label{C9}
Let $\Pi\colon D\to Z$ and 
$\widetilde{\Pi}\colon\widetilde{D}\to\widetilde{Z}$ 
be vector bundles and $\Delta=(\delta,\zeta)$ a 
morphism of vector bundles from $\Pi$ to $\widetilde{\Pi}$. 
If the mapping $\delta\colon D\to\widetilde{D}$ is 
a fiberwise isomorphism, then for every connection map 
$\widetilde{K}$ on the vector bundle $\widetilde{\Pi}$ 
there exists a unique connection map $K$ on the vector bundle $\Pi$ 
such that the diagram 
\begin{equation}\label{D14}
\begin{CD}
TD @>{T\delta}>> T\widetilde{D} \\
@V{K}VV @VV{\widetilde{K}}V \\
D @>{\delta}>> \widetilde{D}
\end{CD}
\end{equation}
is commutative. 
\end{proposition}

\begin{definition}\label{C10}
\normalfont
In the setting of Proposition~\ref{C9} 
the connection map $K$ is called the \textit{pull-back} 
of $\widetilde{K}$ by $\Delta$ and 
we denote $K:=\Delta^*(\widetilde{K})$. 
Also, if $\gamma$ and $\widetilde{\gamma}$ are the connections 
associated with the connection maps $K$ and $\widetilde{K}$, 
respectively, 
then we say that $\gamma$ is the \textit{pull-back} 
of $\widetilde{\gamma}$ by $\Delta$ and 
we denote $\gamma:=\Delta^*(\widetilde{\gamma})$.
\end{definition}

\begin{proof}[Proof of {\rm Proposition~\ref{C9}}]
Since $\delta\colon D\to\widetilde{D}$ is a fiberwise isomorphism, 
the diagram from the statement is commutative if and only if we define 
\begin{equation}\label{D14.5}
(\forall z\in Z)(\forall v\in T_zE)\quad 
K(v):=(\delta|_{D_z})^{-1}(\widetilde{K}(T\delta(v))).
\end{equation}
Thus we get a map $K\colon TE\to E$ that satisfies condition~(i) 
in Theorem~\ref{caractVECTOR}(3). 
It remains to show that condition~(ii) 
in Theorem~\ref{caractVECTOR}(3) is satisfied as well and that $K$ is smooth. 

To this end recall the embedding 
$$
\varepsilon\colon D\fimes_Z D\to TD, \quad (\xi,\eta)\mapsto 
\dot{c}_{\xi,\eta}(0)
$$
as the vertical subbundle of $TD$, 
and let 
$$
\widetilde{\varepsilon}\colon
\widetilde{D}\fimes_{\widetilde{Z}}\widetilde{D}\to 
T\widetilde{D}, \quad (\widetilde{\xi},\widetilde{\eta})\mapsto 
\dot{c}_{\widetilde{\xi},\widetilde{\eta}}(0)
$$
be the similar map associated with the vector bundle 
$\widetilde{\Pi}\colon\widetilde{D}\to\widetilde{Z}$. 
Since the mapping $\delta$ is fiberwise linear, it follows that 
for arbitrary $(\xi,\eta)\in D\fimes_Z D$ we have 
$$
(\forall t\in{\mathbb R})\quad 
\delta(c_{\widetilde{\xi},\widetilde{\eta}}(t))
=\delta(\widetilde{\xi}+t\widetilde{\eta})
=\delta(\widetilde{\xi})+t\delta(\widetilde{\eta})
=c_{\delta(\widetilde{\xi}),\delta(\widetilde{\eta})}(t),$$
whence $T\delta(\dot{c}_{\xi,\eta}(0))
=\dot{c}_{T\delta(\xi),T\delta(\eta)}(0)$. 
In other words the diagram 
\begin{equation}\label{D15}
\begin{CD}
D\fimes_Z D @>{\delta\times\delta}>> 
   \widetilde{D}\fimes_{\widetilde{Z}}\widetilde{D} \\
@V{\varepsilon}VV @VV{\widetilde{\varepsilon}}V \\
TD @>{T\delta}>> T\widetilde{D}
\end{CD}
\end{equation}
is commutative. 
Then for arbitrary $(\xi,\eta)\in D\fimes_Z D$ we have 
$$
\delta(K(\varepsilon(\xi,\eta)))
\mathop{=}^{\eqref{D14}}
  \widetilde{K}(T\delta(\varepsilon(\xi,\eta)))
\mathop{=}^{\eqref{D15}}
  \widetilde{K}(\widetilde{\varepsilon}(\delta(\xi),\delta(\eta)))
=\delta(\eta).
$$
The last of the above equalities follows since $\widetilde{K}$ 
is a connection map hence satisfies 
condition~(ii) in Theorem~\ref{caractVECTOR}(3). 
Since $\delta\colon D\to\widetilde{D}$ is fiberwise injective, 
we obtain $K(\varepsilon(\xi,\eta))=\eta$, and thus $K$ in turn satisfies 
condition~(ii) in Theorem~\ref{caractVECTOR}(3). 

Finally, we will prove that the map $K$ is smooth 
by computing its expression in local coordinates. 
Let $\textbf{Z}$ be the model space of the manifold $Z$ and  
by $\textbf{E}$ the fiber of $\Pi\colon D\to Z$, 
and let $\varphi\colon U\to\textbf{Z}$ be a local chart of $Z$ 
and $\Psi\colon\Pi^{-1}(U)\to\varphi(U)\times\textbf{Z}$ 
a local trivialization of the bundle $\alpha$. 
Similarly, let $\widetilde{\textbf{Z}}$ 
be the model space of the manifold $\widetilde{Z}$,   
$\widetilde{\textbf{E}}$ be the fiber of 
$\widetilde{\Pi}\colon\widetilde{D}\to\widetilde{Z}$, 
$\widetilde{\varphi}\colon\widetilde{U}\to\widetilde{\textbf{Z}}$ 
be a local chart of $\widetilde{Z}$,  
and $\widetilde{\Psi}\colon\widetilde{\Pi}^{-1}(\widetilde{U})\to
\widetilde{\varphi}(\widetilde{U})\times\widetilde{\textbf{E}}$ 
a local trivialization of the bundle $\widetilde{\Pi}$. 
Also denote by 
$\widetilde{\Gamma}\colon\widetilde{U}\to 
\Bc(\widetilde{\textbf{Z}},\widetilde{\textbf{E}};\widetilde{\textbf{E}})$ 
the corresponding Christoffel symbol of $\widetilde{K}$. 
Since the mapping $\zeta\colon Z\to\widetilde{Z}$ is continuous, 
we may assume that $U$ is small enough such that 
$\zeta(U)\subseteq\widetilde{U}$. 

Since the pair $\Delta=(\delta,\zeta)$ is a morphism of vector bundles, 
it follows that there exists a smooth mapping 
$d\colon U\to\Bc(\textbf{E},\widetilde{\textbf{E}})$ 
such that 
\begin{equation}\label{D16}
(\widetilde{\Psi}\circ\delta\circ\Psi^{-1})(x,\xi)
=(\zeta_{\phi,\widetilde{\phi}}(x),\,d(x)\xi)
\quad \text{ whenever }(x,\xi)\in U\times\textbf{E}.
\end{equation}
Using the notation $K_{\Psi}=\Psi\circ K\circ T\Psi^{-1}$, 
$\widetilde{K}_{\widetilde{\Psi}}
=\widetilde{\Psi}\circ\widetilde{K}\circ T\widetilde{\Psi}^{-1}$, 
and 
$\zeta_{\phi,\widetilde{\phi}}=\widetilde{\phi}\circ\zeta\circ\phi^{-1}$ 
we obtain 
\allowdisplaybreaks
\begin{align}
(\zeta_{\phi,\widetilde{\phi}}(x),\,d(x)K_{\Psi}(x,\xi,y,\eta))
&\mathop{=}^{\eqref{D16}}
 (\widetilde{\Psi}\circ\delta\circ\Psi^{-1})(x,K_{\Psi}(x,\xi,y,\eta)) \nonumber\\
&\mathop{=}^{\eqref{D14}}
 \widetilde{K}_{\widetilde{\Psi}}
((T\widetilde{\Psi}\circ T\delta\circ T\Psi^{-1})(x,\xi,y,\eta)) \nonumber\\
&\mathop{=}^{\eqref{D16}}
 \widetilde{K}_{\widetilde{\Psi}}
 (\zeta_{\phi,\widetilde{\phi}}(x),\,d(x)\xi,\, 
    (\zeta_{\phi,\widetilde{\phi}})'(x)y,\,d(x)\eta) \nonumber\\
&\mathop{=}^{\eqref{D13}}
 (\zeta_{\phi,\widetilde{\phi}}(x),\,
 d(x)\eta+\widetilde{\Gamma}(\zeta_{\phi,\widetilde{\phi}}(x))
 ((\zeta_{\phi,\widetilde{\phi}})'(x)y,d(x)\xi)), \nonumber
\end{align}
hence for all 
$x\in\varphi(U)\subseteq\textbf{Z}$, 
$y\in\textbf{Z}$, and 
$\xi,\eta\in\textbf{E}$ we get 
\begin{equation}\label{D17}
K_{\Psi}(x,\xi,y,\eta)=\eta
+d(x)^{-1}\widetilde{\Gamma}(\zeta_{\phi,\widetilde{\phi}}(x))
 ((\zeta_{\phi,\widetilde{\phi}})'(x)y,d(x)\xi). 
\end{equation}
This expression of $K$ in local coordinates shows 
that $K$ is indeed a smooth map. 
\end{proof}

\begin{remark}
\normalfont
By using Theorem~\ref{caractVECTOR} $(2)\Leftrightarrow(3)$, formula~\eqref{D17} and the computation 
in \cite[page 172]{El67} 
we obtain an alternative proof of the fact that the mapping $K\colon TE\to E$ 
defined by~\eqref{D14.5} is a connection map. 
A special case of this reasoning can be found at the end of \cite[Sect. 2]{El67}.  
\end{remark}

\begin{remark}\label{deriv_prop1}
\normalfont
It is possible to establish a result for pull-backs of covariant derivatives as follows. Let $\Pi\colon D\to Z$ and $\widetilde{\Pi}\colon\widetilde{D}\to\widetilde{Z}$ be vector bundles 
endowed with the linear connections $\Phi$ and $\widetilde{\Phi}$, 
with the corresponding covariant derivatives $\nabla$ and $\widetilde{\nabla}$, 
respectively. 
Assume that $\Theta=(\delta,\zeta)$ is a homomorphism of vector bundles from $\Pi$ into $\widetilde{\Pi}$ 
such that $T\delta\circ\Phi=\widetilde{\Phi}\circ T\delta$. 
If $\sigma\in\Omega^0(Z,D)$ and $\widetilde{\sigma}\in\Omega^0(\widetilde{Z},\widetilde{D})$ are 
such that $\delta\circ\sigma=\widetilde{\sigma}\circ\zeta$, 
then $\delta\circ K\circ T\sigma=\widetilde{K}\circ T\widetilde{\sigma}\circ T\zeta$, that is, $\delta\circ\nabla\sigma=\widetilde{\nabla}\widetilde{\sigma}\circ T\zeta$; see details in \cite{BG14}. 
\end{remark}

\section{Stiefel bundles on flag manifolds in Banach algebras}\label{Sect3}

Let $\Ac$ be a complex Banach algebra with unit $\1$ and endowed with a continuous involution $a\mapsto a^*$. 
Let $\Ac^\times$ denote the open set of invertible elements of $\Ac$. We will assume in all of this section that $\Ac$ 
is also {\it hermitian}; that is, $\sigma(a)\subset\R$ whenever $a=a^*\in\Ac$. Here, $\sigma(a)$ is the spectrum of $a$. 
Examples of hermitian Banach algebras are the $C^*$-algebras and group algebras on many locally compact groups.

Put $\Pc(\Ac):=\{p\in\Ac: p=p^2\}$ and $\widehat p=\1-p$, for every $p\in\Pc(\Ac)$. 
For $p,q\in\Pc(\Ac)$, the notation $p\le q$ means that $qp=p$, and $p<q$ means that $p\le q$ with $p\not=q$. 
An equivalence relation $\sim$ can be defined in $\Pc(\Ac)$ by
$$ 
p\sim q \Leftrightarrow p\le q \hbox{ and } q\le p.
$$
The corresponding quotient set is denoted by $\Gr(\Ac):=\Pc(\Ac)/\sim$ and 
for arbitrary $p\in\Pc(\Ac)$ we denote its equivalence class by $[p]\in\Gr(\Ac)$. 
The assumption that $\Ac$ is a hermitian $*$-algebra implies that the subset of orthogonal projections  
$$\Pc_\perp(\Ac):=\{p\in\Pc(\Ac)\mid p=p^*\}$$ 
is a cross-section of the above equivalence relation, 
that is, every element of $\Pc(\Ac)$ is equivalent to exactly one element of $\Pc_\perp(\Ac)$ 
(see \cite[Lemma 4.4(2)]{BN10}). 
One thus obtains the canonical bijection 
\begin{equation}\label{straight}
\Pc_\perp(\Ac)\to\Gr(\Ac),\quad p\mapsto[p].
\end{equation} 
The left action 
$g\cdot[p]:=[gpg^{-1}]$ of $g\in\Ac^\times$ on $[p]\in\Gr(\Ac)$ is well defined. 
Then the {\it Grassmann manifold}, or just {\it Grassmannian}, $\Gr(p,\Ac)$ at $p\in\Pc(\Ac)$ is the complex homogeneous space 
$$
\Gr(p,\Ac):=\{[gpg^{-1}]: g\in\Ac^\times\} \simeq\Ac^\times/ \Delta(p)^\times
$$
where 
$$\Delta(p)^\times
:=\{g\in \Ac^\times: [gpg^{-1}]=[p]\}
=\{g\in \Ac^\times: \widehat pgp=\widehat pg^{-1}p=0\}$$ 
is a Banach-Lie subgroup of $\Ac^\times$. 
As a Banach manifold, $\Gr(p,\Ac)$ is modeled on its tangent space $\widehat p\Ac p$ at $p$. 

For the following facts we refer the reader to \cite{DG01}; see also 
\cite{MS97}, \cite{BN10}. 
The Stiefel bundle on $\Gr(p,\Ac)$ is the principal bundle 
$\sigma\colon\Vc(p,\Ac)\to\Gr(p,\Ac)$ 
whose total space 
is the complex Banach manifold 
$$
\begin{aligned}
\Vc(p,\Ac):=
&\{v\in\Ac: (\exists a\in\Ac)\ ava=a,\ vav=v,\ av=p,\ va\in\Gr(p,\Ac)\}.  
\end{aligned}
$$
If $v\in\Vc(p,\Ac)$, then $\sigma(v)\in\Gr(p,A)$ is defined by 
$[p]=\sigma(v)$, where $p\in\Pc_\perp(\Ac)$ is uniquely determined via the bijection~\eqref{straight} by the conditions 
$$
p=p^2=p^*,\ (va)p=p,\text{ and }p(va)=va. 
$$  
Then $\Vc(p,\Ac)$ is acted on the right by the group $(p\Ac p)^\times$, which turns out to be the structure group of the principal bundle $\sigma$, so that 
$\Gr(p,\Ac)\simeq\Vc(p,\Ac)/(p\Ac p)^\times$. 
Moreover, the holomorphic action $\Ac^\times\times\Vc(p,\Ac)\to\Vc(p,\Ac)$ given by multiplication is transitive, see \cite[Lemma 5.1 and Prop. 5.3]{DG01}. Hence it yields the identity
$\Vc(p,\Ac)=\{ g p:g\in\Ac^\times\}$ and the  biholomorphic $\Ac^\times$-equivariant diffeomorphism 
$\Vc(p,\Ac)\simeq\Ac^\times/\Ac_p^\times$, where
$\Ac_p^\times:=\{g\in\Ac^\times:gp=p\}$ is the isotropy group at $p$. In these forms, $\Vc(p,\Ac)$ is suitably handled.

We next generalize the Stiefel bundle by replacing the Grassmann manifold with flag manifolds.

\subsection{Flag manifolds}

For every integer $n\ge1$, let
$\delta=(p_1,\dots,p_n)$ be a finite, 
totally ordered family 
$0=p_0<p_1<\dots<p_n<p_{n+1}=\1$
of elements in $\Pc(\Ac)$. 
Set 
$\widehat\delta:=(\widehat p_n,\dots\widehat p_1)$.

Let $\Pc_n(\Ac)$ be the set of all such families. Then we define
$$
(p_1,\dots,p_n)\sim(q_1,\dots,q_n)\Longleftrightarrow [p_j]=[q_j]\, (j=1,\dots,n).
$$
Then the above is an equivalence relation. 
The corresponding quotient set is denoted by $\Fl_{\Ac}(n)$, 
and the equivalence class of $\delta=(p_1,\dots,p_n)\in\Pc_n(\Ac)$ is denoted by $[\delta]\in\Fl_{\Ac}(n)$. 
Generalizing the case of Grassmannians, $\Ac^\times$ acts on $\Pc_n(\Ac)$ by
$g\cdot[(p_1,\dots,p_n)]:=[(gp_1g^{-1},\dots,gp_ng^{-1})]$, for $g\in\Ac^\times$, 
$[(p_1,\dots,p_n)]\in\Fl_{\Ac}(n)$. 
Then, for fixed $[\delta]\in\Fl_{\Ac}(n)$, the {\it flag manifold} at $[\delta]$ is the $\Ac^\times$-orbit
$$
\Fl_{\Ac}([\delta]):=\{g\cdot [\delta]:g\in\Ac^\times\}
$$
Clearly, $\Fl_{\Ac}([\delta])=\Gr(p,\Ac)$ when $n=1$ and $p=p_1$. 
Similarly to that case, 
we have the $\Ac^\times$-equivariant bijection
$$
\Fl_{\Ac}([\delta])\simeq\Ac^\times/\Delta(\delta)^\times,
$$
where $\Delta(\delta)^\times$ is the group of invertible elements of the non-self-adjoint subalgebra of $\Ac$,
$$
\Delta(\delta):=\{a\in\Ac:\widehat p_jap_j=0\ (j=1,\dots,n)\}.
$$ 
In fact, the above $\Ac^\times$-equivariant bijection is a diffeomorphism of Banach manifolds where the topology in $\Fl_{\Ac}([\delta])$ coincides with the quotient topology induced from $\Ac$. 
This is proven in \cite{BN10}, where the following notions are introduced accordingly.

Let $\Phi_\delta\colon\Ac\to\Ac$ be the {\it diagonal truncation} on $\Ac$ defined by
$$
\Phi_\delta(x):=\sum_{j=1}^{n+1}(p_j-p_{j-1})x(p_j-p_{j-1}).
$$
Then it is readily seen that 
$\Phi_\delta$ is a continuous idempotent mapping with range 
equal to 
$D(\delta):=\{x\in\Ac:xp_j=p_jx\ (j=1,\dots,n)\}$, and that  $\Phi_\delta=\Phi_{\widehat\delta}$. 
Also, the restriction of $\Phi_\delta$ to $\Delta(\delta)$ is multiplicative. 

Denote 
$N(\delta):=\Delta(\delta)^\times\cap(\Phi_\delta)^{-1}(\1)$, 
which is a Banach-Lie subgroup of $\Delta(\delta)^\times$. 
The flag manifold $\Fl_{\Ac}([\delta])$ is modeled on the tangent space of $N(\widehat\delta)$. 
This fact is obtained from the unique decomposition
\begin{equation}\label{decomposition}
\Omega_\delta=N(\widehat\delta)\Delta(\delta)^\times
\left(=N(\widehat\delta)D(\delta)^\times N(\delta)\right)
\end{equation}
given by \cite[Rem. 4.2]{BN10}, together with some other facts, 
(see \cite[Lemma A1, Th. 4.3]{BN10}). 
We note that the matrix decomposition given in \cite[Rem. 4.2]{BN10}, 
should actually read as follows: 
for $g\in\Ac^\times$ with $pgp\in(p\Ac p)^\times$ and $p\in\Pc(\Ac)$, one has 
$$g=\begin{pmatrix}
p & 0 \\
\widehat pg(pgp)^{-1} & \widehat p
\end{pmatrix}
\begin{pmatrix}
pgp & 0 \\
0 & g-g(pgp)^{-1}g
\end{pmatrix}
\begin{pmatrix}
p & (pgp)^{-1}g\widehat p \\
0 & \widehat p
\end{pmatrix}
$$
where $(pgp)^{-1}$ is the inverse of $pgp$ in $(p\Ac p)^\times$. 
(For a general $n$-uple $\delta$ one proceeds by induction; see \cite{BN10}.)

The above decomposition \eqref{decomposition} also 
shows that 
$\Ac^\times\to\Fl_{\Ac}([\delta])$ is a principal bundle with structure group $\Delta(\delta)^\times$. 
Here, 
$\Omega_\delta
:=\{g\in\Ac^\times:p_jgp_j\in(p_j\Ac p_j)^\times \hbox{ for } j=1,\dots,n\}$ is an open subset of $\Ac^\times$ with $\1\in\Omega_\delta$. 
(Note that the unit of $p\Ac p$ is $p$, for all $p\in\Pc(\Ac)$.) 

From \eqref{decomposition} one directly obtains 
\begin{equation}\label{directsum}
\Ac=TN(\widehat\delta)\dot+\Delta(\delta)
=TN(\widehat\delta)\dot+D(\delta)\dot+ TN(\delta)
\end{equation}
where $\Ac$, $TN(\widehat\delta)$, $\Delta(\delta)$, $TN(\delta)$ are seen as the tangent spaces (at \1) of 
the groups $\Ac^\times$, $N(\widehat\delta)$, $\Delta(\delta)^\times$, $N(\delta)$ respectively.

Let us now have a closer look at \eqref{directsum}. 
First, we describe the elements of $N(\delta)$ in some detail. 

\begin{lemma}\label{Ndelta}
In the above setting,
$$
g\in N(\delta)\Longleftrightarrow 
\widehat p_{j-1}gp_j=\widehat p_{j-1}p_j\ (j=1,\dots,n, n+1).
$$
\end{lemma}

\begin{proof} 
Clearly, $p_j-p_{j-1}=\widehat p_{j-1}p_j$ for every $j=1,\dots,n,n+1$. For 
$x\in(\Phi_\delta)^{-1}(\1)$ one has 
$\sum_{j=1}^{n+1}(p_j-p_{j-1})x(p_j-p_{j-1})=\1$,
whence, multiplying the two members of the equality by 
$\widehat p_{k-1}p_k$ with fixed $k$ in $\{1,\dots,n,n+1\}$, one eventually obtains
$$ 
x\in(\Phi_\delta)^{-1}(\1)\Longleftrightarrow 
\widehat p_{k-1}p_kxp_k\widehat p_{k-1}=\widehat p_{k-1}p_k\ (k=1,\dots,n, n+1). 
$$

Moreover, if $x\in\Delta(\delta)$ then 
$$
\begin{aligned}
\widehat p_{k-1}p_kxp_k\widehat p_{k-1}=\widehat p_{k-1}xp_k\widehat p_{k-1}
&=\widehat p_{k-1}xp_k-\widehat p_{k-1}xp_{k-1}=\widehat p_{k-1}xp_k
\end{aligned}
$$
and the statement follows. 
\end{proof}

\begin{remark}
\normalfont
Lemma \ref{Ndelta} is equivalent to 
$$
g\in N(\widehat\delta)\Longleftrightarrow 
p_jg\widehat p_{j-1}=p_j\widehat p_{j-1}\ (j=1,\dots,n+1).
$$

It follows from the characterization that the tangent space $TN(\widehat\delta)$ of the Banach-Lie group $N(\widehat\delta)$ is given by
$$
a\in TN(\widehat\delta)\Longleftrightarrow p_ja\widehat p_{j-1}=0\ (j=1,\dots,n+1).
$$

In particular, for $n=1$, $\delta: 0<p<\1$ and 
$N(\widehat p):=N(\widehat\delta)$ we have that $a\in T_\1N(\widehat p)$ if only if $pa=0$ and $a\widehat p=0$; that is, $T_\1N(\widehat p)=\widehat p\Ac p$, the space on which the Grassmannian $\Gr(p,\Ac)=\Fl_{\Ac}([p])$ is modeled.
\end{remark}

\begin{lemma}\label{tangentN}  
One has the direct sum decomposition 
$$
\Phi_{\delta}^{-1}(0)=TN(\widehat\delta)\dot + TN(\delta)
$$
with
$$
TN(\widehat\delta)=\left(\Phi_{\delta}^{-1}(0)\cap\Delta(\widehat\delta)\right)
 \hbox{ and } TN(\delta)=\left(\Phi_{\delta}^{-1}(0)\cap\Delta(\delta)\right). 
$$
In fact, for every $x\in\Phi_{\delta}^{-1}(0)$, one has $x=y+z$ where 
$$
y=\sum_{j=2}^{n+1}\widehat p_{j-1}p_jxp_{j-1}\in TN(\widehat\delta) 
\hbox{ and }
z=\sum_{j=2}^{n+1} p_{j-1}xp_j\widehat p_{j-1}\in TN(\delta).
$$
\end{lemma}

\begin{proof}
Notice that $x\in\Phi_{\delta}^{-1}(0)$ if and only if 
$$
\widehat p_{j-1}(p_jxp_j)\widehat p_{j-1}=0,\ \forall j=1,\dots,n,n+1
$$
if and only if 
$$
\begin{aligned}
p_jxp_j&=p_{j-1}(p_jxp_j)p_{j-1}+p_{j-1}(p_jxp_j)\widehat p_{j-1}\\
&=
p_{j-1}xp_{j-1}+p_{j-1}xp_j\widehat p_{j-1}+\widehat p_{j-1}p_jxp_{j-1}.
\end{aligned}
$$

Therefore,
$$
\begin{aligned}
x&=p_{n+1}xp_{n+1}=p_nxp_n+p_nxp_{n+1}\widehat p_n+\widehat p_np_{n+1}xp_n\\
&=p_{n-1}xp_{n-1}+\sum_{j=n}^{n+1}p_{j-1}xp_j\widehat p_{j-1}
+\sum_{j=n}^{n+1}\widehat p_{j-1}p_jxp_{j-1}\\
&=\dots=\\
&=p_1xp_1+\sum_{j=2}^{n+1}p_{j-1}xp_j\widehat p_{j-1}
+\sum_{j=2}^{n+1}\widehat p_{j-1}p_jxp_{j-1}\\
&=\sum_{j=2}^{n+1}p_{j-1}xp_j\widehat p_{j-1}
+\sum_{j=2}^{n+1}\widehat p_{j-1}p_jxp_{j-1}
\end{aligned}
$$

Also, $x\in\Phi_{\delta}^{-1}(0)\cap\Delta(\delta)$ if and only if 
$p_jxp_j=p_{j-1}xp_{j-1}+p_{j-1}xp_j\widehat p_{j-1}+\widehat p_{j-1}p_jxp_{j-1}$ 
and 
$\widehat p_{j-1}xp_j=0$ for every $j=1,\dots,n,n+1$. This is the same as 
$xp_j=
xp_{j-1}+p_{j-1}xp_j-xp_{j-1}+xp_{j-1}-xp_{j-1}=p_{j-1}xp_j$; that is, 
$\widehat p_{j-1}xp_j=0$ for all $j=1,\dots,n,n+1$. Equivalently, $x\in TN(\delta)$.

Similarly, $x\in\Phi_{\delta}^{-1}(0)\cap\Delta(\widehat\delta)$ if and only if $p_jx\widehat p_{j-1}=0$ ($j=1,\dots,n,n+1$), 
if and only if 
$x\in TN(\widehat\delta)$.

Fix $k\in\{1,\dots,n\}$. Then, 
$$
\begin{aligned}
\widehat p_{k-1}zp_k
&=\widehat p_{k-1}\left(\sum_{j=2}^{n+1} p_{j-1}xp_j\widehat p_{j-1}\right)p_k
=\sum_{j=2}^{n+1}\widehat p_{k-1}p_{j-1}xp_jp_k\widehat p_{j-1}\\
&=\sum_{j=2}^{k} (\widehat p_{k-1}p_{j-1})xp_j\widehat p_{j-1}
+\sum_{j=k+1}^{n+1}\widehat p_{k-1}p_{j-1}x(p_k\widehat p_{j-1})=0+0=0, 
\end{aligned}
$$
whence $z\in\Delta(\delta)$. Analogously, $p_ky\widehat p_{k-1}=0$ and therefore $y\in\Delta(\widehat\delta)$.

Finally, it is clear that $x\in\left(\Phi_{\delta}^{-1}(0)\cap\Delta(\widehat\delta)\right)
\bigcap\left(\Phi_{\delta}^{-1}(0)\cap\Delta(\delta)\right)$ if and only if $x=0$.
\end{proof}

\begin{remark}\label{connectionE}
\normalfont
Lemma \ref{tangentN} contributes to understand the decomposition \eqref{decomposition} 
in relation with that one defined by $\Phi_\delta$,
$$
\begin{aligned}
\Ac&
=TN(\widehat\delta)\dot +\Delta(\delta)
=TN(\widehat\delta)\dot + TN(\delta)\dot+D(\delta)\\
&
=\left(\Phi_{\delta}^{-1}(0)\cap\Delta(\widehat\delta)\right)\dot
+\left((\Phi_{\delta}^{-1}(0)\cap\Delta(\delta))\dot+D(\delta)\right)
=\Ker\Phi_\delta\dot+\Ran\Phi_\delta,
\end{aligned}
$$
and allows us to find the associated projection 
$E_\delta\colon\Ac\to\Delta(\delta)$ defined by
$$
E_\delta(a):=\sum_{j=1}^{n+1}\widehat p_{j-1}p_j a \widehat p_{j-1}, 
\quad a\in\Ac.  
$$
As a matter of fact, $E_\delta$ is a projection of $\Ac$ onto 
$\Delta(\delta)$, with $\Ker E_\delta= TN(\widehat\delta)$, such that 
$E_\delta(\1)=\1$. 
The projection $E_\delta$ and its image $\Delta(\delta)$ hold a central role in endowing the flag manifold at $[\delta]$ 
with the structure of a complex manifold on which the unitary group of $\Ac$ acts by holomorphic maps; 
this can be done just as in the case of Grassmann manifolds (see for instance \cite[Sect. 3]{BG09}). 

Put $\Ac^\times_{\delta}:=\{g\in\Ac^\times:gp_j=p_j\ (j=1,\dots,n)\}
=\{g\in\Ac^\times:gp_n=p_n\}$ and 
$\Gc([\delta]):=\Delta(\delta)^\times/\Ac^\times_{\delta}$, where the quotient is understood in the sense of the multiplication in $\Ac$. 
\end{remark}

\begin{lemma}\label{normalG} 
The following assertions hold true: 
\begin{enumerate}[(i)]
\item\label{normalG_item1}
$\Ac^\times_{\delta}$ is a normal subgroup of $\Delta(\delta)^\times$.
\item\label{normalG_item2}
The multiplication in $\Ac^\times$ induces a natural free right action of the group $\Gc([\delta])$ on the homogeneous space 
$\Ac^\times/\Ac^\times_{\delta}$, and this action is holomorphic.  
\end{enumerate}
\end{lemma}
 
\begin{proof}
\eqref{normalG_item1} This is straightforward.

\eqref{normalG_item2} The action referred to in the statement is 
$$
\alpha\colon\left(\Ac^\times/\Ac^\times_{\delta}\right)\times\Gc([\delta])
\to\Ac^\times/\Ac^\times_{\delta},
\quad (g\Ac^\times_{\delta},a\Ac^\times_{\delta})\mapsto
(ga)\Ac^\times_{\delta}. 
$$  
Let $g,h\in\Ac^\times$ and $a,b\in\Delta(\delta)^\times$ such that 
$g\Ac^\times_{\delta}=h\Ac^\times_{\delta}$ and 
$a\Ac^\times_{\delta}=b\Ac^\times_{\delta}$. Then, for $j=1,\dots,n$, 
$$
b^{-1}h^{-1}gap_j=b^{-1}(h^{-1}g)ap_j=b^{-1}p_jap_j=b^{-1}ap_j=p_j,
$$
whence 
$ga\Ac^\times_{\delta}=hb\Ac^\times_{\delta}$. Thus the action $\alpha$ is well defined. That $\alpha$ is free is even simpler to show, and the holomorphy of $\alpha$ follows by usual arguments.
\end{proof}

\begin{remark}
\normalfont
The mapping $g\Ac^\times_{\delta}\mapsto (p_jgp_j)_{j=1}^n=(gp_j)_{j=1}^n$ is a biholomorphic group homomorphism between the Banach-Lie groups 
$\Gc([\delta])$ and 
$\{(h_1,\dots,h_n):h_j\in(p_j\Ac p_j)^\times, h_j=h_np_j\ (j=1,\dots,n))\}$. When $n=1$, one gets that the structure group $(p\Ac p)^\times$ of the Stiefel bundle $\sigma$ is isomorphic to the quotient group
$\Gc([p])=\Delta(p)^\times/\Ac^\times_{p}$.
\end{remark}

\begin{proposition}\label{commonFl} 
The  
homogeneous spaces  
$$
\Fl_\Ac([\delta]) 
\text{ and }
\left(\Ac^\times/\Ac^\times_\delta\right)/\Gc([\delta])
$$ 
are biholomorphically diffeomorphic.
\end{proposition}

\begin{proof}
Since $\Ac^\times_\delta\subset\Delta(\delta)^\times$ the quotient map 
$\Ac^\times/\Ac^\times_\delta\to\Ac^\times/\Delta(\delta)^\times$ is well defined and surjective. Moreover, $g\Ac^\times\mapsto\Delta(\delta)^\times$ if and only if 
$g\in\Delta(\delta)^\times$. The remainder of the proof is standard. 
\end{proof}

As said before, the multiplication in $\Ac$ splits the open set $\Omega_\delta$ as
$$   
\Omega_\delta=N(\widehat\delta)\Delta(\delta)^\times.
$$
A similar factorization holds for the open set $\Omega_\delta\Ac^\times_\delta$ of 
the homogeneous space $\Ac^\times/\Ac^\times_\delta$ under the natural left action of the group $N(\widehat\delta)$ on the group 
$\Gc([\delta])$.

\begin{proposition}\label{factor}
Set $\widetilde\Omega_\delta:=\{g\Ac^\times_\delta:g\in\Omega_\delta\}$. Then
$$
\widetilde\Omega_\delta
=N(\widehat\delta)\cdot \Gc([\delta]).
$$
and this factorization is unique for every $g\Ac^\times_\delta$ in 
$\widetilde\Omega_\delta$.
\end{proposition}

\begin{proof}
It is readily seen that the action
$$
N(\widehat\delta)\times\Gc([\delta])
\to\widetilde\Omega_\delta,\quad (a,g\Ac^\times_\delta)\mapsto (ag)\Ac^\times_\delta
$$
is well defined and surjective. Moreover, if $a,b\in N(\widehat\delta)$ and $g,h\in\Delta(\delta)^\times$ are such that 
$a\cdot g\Ac^\times_\delta=b\cdot h\Ac^\times_\delta$ then  
$b^{-1}agp_n=hp_n$, with $gp_n, hp_n\in\Delta(\delta)^\times$. By the uniqueness of decompositions in 
$\Omega_\delta=N(\widehat\delta)\Delta(\delta)^\times$ one has $b^{-1}a=\1$ and therefore $a=b$ and $g\Ac^\times_\delta=h\Ac^\times_\delta$. Thus the factorization is unique. 
\end{proof}

\subsection{Stiefel manifold}

In Stiefel bundles on Grassmannians, the action of $\Ac^\times$ on the total space $\Vc(p,\Ac)$ is transitive, 
so that $\Vc(p,\Ac)=\Ac^\times p$. 
Then the bundle map $\sigma\colon\Vc(p,\Ac)\to\Gr(p,\Ac)$ gets the simple form $gp\mapsto[gpg^{-1}]$. 
Similarly, for arbitrary $n$ and $\delta$, 
one looks for a Stiefel bundle with total space  $\Ac^\times\cdot(p_1,\dots,p_n)$ and bundle map 
$(gp_1,\dots,gp_n)\mapsto[(gp_1g^{-1},\dots,gp_ng^{-1})]$.

For $\delta=(p_1,\dots,p_n)\in\Pc_n(\Ac)$, let $\Vc(\delta,\Ac)$ denote the set of 
$n$-uples given by
$$
\begin{aligned}
&\{(v_1,\dots,v_n)\in\Ac^n:v_j\in\Vc(p_j,\Ac)\  
(1\le j\le n), v_j=v_{j+1}p_j\ (1\le j\le n-1)\}\\
& =\{(vp_1,\dots,vp_{n-1},v)\in\Ac^n:v\in\Vc(p_n,\Ac)\}\\
& =\{(gp_1,\dots,gp_n)\in\Ac^n:g\in\Ac^\times\}=:\Vc(\delta,\Ac).
\end{aligned} 
$$
There is a natural bijection from $\Vc(\delta,\Ac)$ onto $\Vc(p_n,\Ac)$. 
We thereby identify both these sets, their topologies and differential structures. 
There exist $\Ac^\times$-equivariant bijections  
$\Vc(\delta,\Ac)\simeq \Ac^\times/\Ac^\times_\delta
\simeq\Ac^\times/\Ac^\times_{p_n}$.

\begin{definition}\label{Nstiefel}
\normalfont
We call \emph{Stiefel bundle on the flag manifold} $\Fl([\delta])$ the map
$$
\sigma_\delta\colon\Vc(\delta,\Ac)\to\Fl_\Ac([\delta]), 
\quad gp_n\mapsto[(gp_1g^{-1},\dots,gp_ng^{-1})] 
$$
\end{definition}

Since $p_np_j=p_j$ for every $j=1,\dots,n$, it is easily checked that the above map~$\sigma_\delta$ is well defined. 

\begin{proposition}\label{Stiefel}
The Stiefel bundle $\sigma_\delta\colon\Vc(\delta,\Ac)\to\Fl([\delta])$ is a principal bundle with structure group $\Gc([\delta])$. 
\end{proposition}

\begin{proof}
To begin with, the right action of $\Gc([\delta])$ on $\Vc(\delta)$ given in 
Lemma \ref{normalG} is free as said there.  

Let $\pi_\delta\colon\Ac^\times\to\Ac^\times/\Delta(\delta)^\times=\Fl_\Ac([\delta])$ 
be the canonical quotient map.  
Define $\psi_1
=\pi_{\delta\mid_{N(\widehat\delta)}}\colon N(\widehat\delta)\to\Fl_\Ac([\delta])
=\Vc(\delta)/\Gc([\delta])$, and then set 
$V_1=\psi_1(N(\widehat\delta))$, $\psi_g=g\cdot\psi_1$, $V_g=g\cdot V_1$ ($g\in\Ac^\times$).

Proposition \ref{factor} implies that the multiplication mapping 
$$
N(\widehat\delta)\times\Gc([\delta])\to\widetilde\Omega_\delta,
\quad (g,a\Ac^\times_\delta)\mapsto (ga)\Ac^\times_\delta
$$
is a homeomorphism, since the map 
$N(\widehat\delta)\times\Delta(\delta)^\times\to\Omega_\delta$,
$(g,a)\mapsto ga$, is a homeomorphism (\cite[Remark 4.2]{BN10}) and the canonical projections  
$N(\widehat\delta)\times\Delta(\delta)^\times
\to N(\widehat\delta)\times\Gc([\delta])$ and $\Omega_\delta\to\widetilde\Omega_\delta$ 
are open.

Then the $\Gc([\delta])$-equivariant trivialization diagram 
\begin{equation*}
\begin{CD}
\widetilde\Omega @>{\simeq}>> N(\widehat\delta)\times\Gc([\delta])
@>{\psi_1\times\1_\Gc}>> V_1\times\Gc([\delta]) \\
@VV{}V @VV{}V @VV{}V \\
N(\widehat\delta) @>{\simeq}>> N(\widehat\delta) @>{\Psi_1}>> V_1
\end{CD}
\end{equation*} 
and arguments involving the family 
$(\psi_g\colon N(\widehat\delta)\to V_g)_{g\in\Ac^\times}$, similar to some reasoning in the proof of \cite[Lemma A.1]{BN10}, shows that $\sigma_\delta$ is a 
$\Gc(\delta)$-principal bundle as required. 
\end{proof}

Now, notice that the tangent space $T\Vc(\delta,\Ac)$ is 
$$
 T\Vc(\delta,\Ac)
 =T\left(\Ac^\times/\Ac^\times_\delta\right)\simeq\Ac/\Ac_\delta
=\Ac/\{ap_n=0\}\simeq\Ac p_n.
$$
and
$$
x\in TN(\widehat\delta)\Leftrightarrow p_jx\widehat p_{j-1}=0\ (j=1,\dots,n+1)
\Rightarrow x\widehat p_n=0
\Leftrightarrow x=xp_n,
$$
whence
$TN(\widehat\delta)=TN(\widehat\delta)p_n$.

Hence, $\Ac p_n=TN(\widehat\delta)\dot +\Delta(\delta)p_n$ and the mapping 
$\widetilde E_\delta:=E_\delta\mid_{\Ac p_n}$ is a smooth projection of 
$\Ac p_n$ onto $\Delta(\delta)p_n$.

\subsection{Connections on Stiefel and frame bundles over flag manifolds}\label{subsect3.3}

In this subsection we construct natural principal connections on Stiefel bundles over flag manifolds.
To this end we work with the unitary reduction of the bundles considered above, 
and we use the full force of the assumption that $\Ac$ is a hermitian $*$-algebra.  
That assumption is encoded in the fact that its unitary group 
$U(\Ac):=\{u\in\Ac\mid uu^*=u^*u=\1\}$ 
acts transitively on the flag manifolds 
$\Fl_\Ac([\delta])\simeq\Fl(\delta):=\{(up_1u^{-1},\dots, up_nu^{-1})\mid u\in U(\Ac)\}$, 
for all $\delta\in\Pc_n(\Ac)$ such that $p_j\in\Pc_\perp(\Ac)$ for $j=1,\dots,n$, by \cite[Th. 4.5]{BN10}. 
Regarding the unitary group $U(\Ac)$ as a Banach manifold, 
its tangent space is isomorphic to the real Lie subalgebra of $\Ac$ 
defined by $\ug(\Ac):=\{a\in\Ac\mid a^*=-a\}$.

Let $\delta=(p_1,\dots,p_n)\in\Pc_n(\Ac)$ with $p_j\in\Pc_\perp(\Ac)$ for $j=1,\dots,n$. 
The corresponding isotropy group at $[\delta]\in\Fl(\delta)$ is 
$U(\Ac)\cap \Delta(\delta)^\times$ and one thus obtains a $U(\Ac)$-equivariant diffeomorphism 
\begin{equation}\label{unitaryBundle}
U(\Ac)/(U(\Ac)\cap \Delta(\delta)^\times)\simeq \Fl(\delta)
\end{equation}
which gives a principal bundle $U(\Ac)\to \Fl(\delta)$ 
whose structural group is the Banach-Lie subgroup $U(\Ac)\cap \Delta(\delta)^\times$ of $U(\Ac)$.  

Using the hypothesis $p_j^*=p_j$ for $j=1,\dots,n$, we obtain 
$$
U(\Ac)\cap \Delta(\delta)^\times=\{u\in U(\Ac)\mid up_j
=p_ju\text{ for }j=1,\dots,n\}=U(\Ac)\cap D(\delta).
$$
Put $D^U(\delta):=U(\Ac)\cap D(\delta)$ and 
$D^{\ug}(\delta):=\ug(\Ac)\cap D(\delta)$. 
Then the diffeomorphism~\eqref{unitaryBundle} is the same as  
$\U(\Ac)/D^U(\delta)\simeq \Fl(\delta)$, whence we obtain the isomorphism between the tangent manifolds
\begin{equation}\label{unitaryBundle2}
{\ug}(\Ac)/D^{\ug}(\delta)\simeq T\Fl(\delta)\simeq TN(\widehat\delta).
\end{equation}

On the other hand, from the decomposition 
$\Ac=\Ker\Phi_\delta\dot+\Ran\Phi_\delta$ discussed in 
Remark~\ref{connectionE}, one easily obtains  
$$
\ug(\Ac)
=\left(\Ker\Phi_\delta\cap\ug(\Ac)\right)
\dot+\left(\Ran\Phi_\delta\cap\ug(\Ac)\right)
=\left(\Ker\Phi_\delta\cap\ug(\Ac)\right)\dot+D^{\ug}(\delta)
$$
since the map $\Phi_\delta$ preserves the involution in $\Ac$ because 
$p_j\in\Pc_\perp(\Ac)$ for $j=1,\dots,n$. 
Hence, we also have the isomorphism of Banach spaces
\begin{equation}\label{unitaryBundle3}
{\ug}(\Ac)/D^{\ug}(\delta)\simeq\left(\Ker\Phi_\delta\cap\ug(\Ac)\right)
\end{equation}
and then, having together (\ref{unitaryBundle2}) and (\ref{unitaryBundle3}),  we can see that the hermitian character of the algebra $\Ac$ is also encoded in the real Banach isomorphism between $TN(\widehat\delta)$ and 
$\Ker\Phi_\delta\cap\ug(\Ac)$; that is, between 
$TN(\widehat\delta)$ and 
$\left(TN(\widehat\delta)\dot+TN(\delta)\right)\cap\ug(\Ac)$.

\begin{remark}\label{pre-reduction}
\normalfont
Let $D(\delta)_+^\times$ denote the subset of positive invertible elements of $D(\delta)$.
The unique decomposition 
$$
\Ac^\times=U(\Ac)D(\delta)_+^\times N(\delta)
$$
obtained in \cite[Cor. 3.7]{BN10} for {\it hermitian} Banach algebras 
 is the key property underlying the observation prior to this remark. 
In effect, the above factorization, besides giving us the diffeomorphism \eqref{unitaryBundle}, also implies 
the direct sum of tangent spaces  
$$
\Ac=\ug(\Ac)\dot+D_{\rm sym}(\delta)\dot+TN(\delta),
$$ 
where with $D_{\rm sym}(\delta)$ we denote the self-adjoint elements of $D(\delta)$. Then, up to isomorphism,
$$
\Ac=TN(\widehat\delta)\dot+\ug_\delta(\Ac)\dot+D_{\rm sym}(\delta)
\dot+TN(\delta)=TN(\widehat\delta)\dot+TN(\delta)\dot+D(\delta)
$$
as shown in Remark \ref{connectionE}.
\end{remark}

It follows from the discussion preceding the above remark that the map 
$\Phi_\delta\mid_{\ug(\Ac)}$ is a principal connection for the unitary reduced bundle 
$U(\Ac)\to U(\Ac)/D^U(\delta)=\Fl(\delta)$. For, given $g$ in the structure group $D^U(\delta)$ of the bundle and $a\in\ug(\Ac)=TU(\Ac)$,
$$ 
\begin{aligned}
(\Phi_\delta\circ T\mu_g)(a)
&=\sum_{j=1}^{n+1}\widehat p_{j-1}p_j(a g)\widehat p_{j-1}p_j\\&
=\sum_{j=1}^{n+1}\widehat p_{j-1}p_j(a)\widehat p_{j-1}p_j g
=(T\mu_g\circ\Phi_\delta)(a).
\end{aligned}
$$
Analogous properties hold in the case of Stiefel bundles.

For every $j\in\{1,\dots,n\}$, define 
$\Vc^U(p_j,\Ac):=\{up_j\mid u\in U(\Ac)\}\subseteq \Vc(p_j,\Ac)$. Then 
one has the Stiefel bundle 
$\sigma_j\colon\Vc(p_j,\Ac)\to\Gr(p_j,\Ac)$ 
with its unitary reduction $\sigma_j^U\colon\Vc^U(p_j,\Ac)\to\Gr(p_j,\Ac)$ defined by 
$$
\Vc^U(p_j,\Ac):=\{up_j\mid u\in U(\Ac)\}\subseteq \Vc(p_j,\Ac)
$$ 
and $\sigma_j^U:=\sigma_j\vert_{\Vc^U(p_j,\Ac)}$.  
(See for instance \cite[Ch. II]{KN63} for reduction theory of finite-dimensional principal bundles.) 

Let now $\delta=(p_1,\dots,p_n)\in\Pc_n(\Ac)$ with $p_j\in\Pc_{\perp}(\Ac)$ for $j=1,\dots,n$. We define 
$\Vc^U(\delta,\Ac):=\{(up_1,\dots, up_n)\mid u\in U(\Ac)\}\simeq\Vc^U(p_n,\Ac)$, so that there is the diffeomorphism 
$\Vc^U(\delta,\Ac)\simeq U(\Ac)/(\Ac^\times_\delta\cap U(\Ac))
=U(\Ac)/(\{u\in U(\Ac)\mid up_n=p_n\})$. 
On the level of tangent spaces,
$$
T\Vc^U(\delta,\Ac)=\ug(\Ac)/\{a\in\ug(\Ac)\mid ap_n=0\}\simeq\ug(\Ac) p_n.
$$
Set $\Gc^U(\delta):=D^U(\delta)/\{up_n=p_n\}\simeq D^U(\delta)p_n$, which is a subgroup of $\Gc(\delta)$. 
As in Lemma \ref{normalG}, the group 
$\Gc^U(\delta)$ acts freely, on the right, on the manifold $\Vc^U(\delta,\Ac)$ and that action is smooth. 
Moreover, and similarly to Proposition \ref{commonFl}, 
the manifolds $\Fl(\delta)$ and $\Vc^U(\delta,\Ac)/\Gc^U(\delta)$ are diffeomorphic. 
We call {\it unitary Stiefel bundle} the projection map
$$
\sigma^U\colon\Vc^U(\delta,\Ac)
\to\Vc^U(\delta,\Ac)/ \Gc^U(\delta)\simeq\Fl(\delta).
$$

\begin{proposition}\label{principalUniStiefel}
The unitary Stiefel bundle is a $\Gc^U(\delta)$-principal bundle under the right-action 
$$
\Vc^U(\delta,\Ac)\times \Gc^U(\delta)\to\Vc^U(\delta,\Ac), 
\quad (up_n, ap_n)\mapsto uap_n,
$$
admitting a principal connection $\Phi_\delta^U$ given by 
$$
\Phi_\delta^U\colon\ug(\Ac)p_n\to D^{\ug}(\delta)p_n, 
\quad ap_n\mapsto \Phi_\delta(a)p_n.
$$ 
\end{proposition}

\begin{proof}
Like formerly for $\Phi_{\delta}\vert_{\ug(\Ac)}$, we have that $\Phi_\delta^U$ is 
$\Gc^U(\delta)$-equivariant from $\ug(\Ac)p_n$ to $D^{\ug}(\delta)p_n$. Also, the fact that $\Phi_\delta^U$ is the restriction of the map $\Phi_\delta$, which indeed is defined on all of $\Ac$, and the properties of $\Phi_\delta$ allow us to get suitable sections showing that the projection $\sigma^U$ has local trivializations compatible with the action of $\Gc^U(\delta)$.
\end{proof}

Naturally, the connection $\Phi_\delta^U$ given in Proposition 
\ref{principalUniStiefel} can be translated in terms of its equivalent notions 
as done in Theorem \ref{principalCARACT}, as a connection form in particular. On the other hand, the connection 
$\Phi_\delta^U$ induces linear connections on the vector bundles associated with $\sigma^U$ via $*$-representations of 
$D^U(\delta)p_n$. Next, we consider the case of the tautological vector bundles.

\medskip
Let $\delta\in\Pc_n(\Ac)$ with $p_j\in\Pc_{\perp}(\Ac)$ for $j=1,\dots,n$. Put 
$$
\Tc_\delta(A)
:=\{\left((up_1\Ac,\dots,up_n\Ac), (up_1x_1,\dots,up_nx_n)\right)\mid u\in U(\Ac), x\in\Ac\}
$$
and define the tautological bundle 
$\Pi_\delta\colon \Tc_\delta(A)\to\Fl(\delta)$ by
$$
\Pi_\delta\colon\left((up_j\Ac)_{j=1}^n, (up_jx_j)_{j=1}^n\right)
\mapsto (up_j\Ac)_{j=1}^n\equiv (up_ju^{-1})_{j=1}^n.
$$

On the other hand, under the representation defined by the natural juxtaposition action $D^U(\delta)p_n\times (p_1\Ac\times\dots p_n\Ac)$ one obtains the vector bundle 
$$
\Pi_{\sigma(\delta)}^U\colon\Vc^U(\delta,\Ac)\times
_{D^U(\delta)p_n} 
(p_1\Ac\times\dots p_n\Ac)\to\Fl(\delta)
$$
given by the asignment $(up_n,(p_1x_1,\dots,p_nx_n))\mapsto (up_1\Ac,\dots,up_n\Ac)$. 

Let $\Theta_\Vc\colon\Vc^U(\delta,\Ac)\times
_{D^U(\delta)p_n} 
(p_1\Ac\times\dots p_n\Ac)\to\Tc_\delta(A)$ denote the map
$$
(up_n,(p_1x_1,\dots,p_nx_n))\mapsto
((up_1\Ac,\dots,up_n\Ac),(up_1x_1,\dots,up_nx_n)).
$$
It is readily seen that the pair $(\Theta_\Vc,\id_{\Fl(\delta)})$ is a vector bundle diffeomorphism from $\Pi_{\sigma(\delta)}^U$ onto $\Pi_\delta$. 

Then, in accordance with what was recalled in the subsection~\ref{subs2.4}, the connection induced by $\Phi_\delta^U$ on 
$\Pi_{\sigma(\delta)}^U=\Pi_\delta$ is the map
$$
\widetilde\Phi_\delta\colon 
\ug(\Ac)p_n\times_{D^U(\delta)p_n}T\Pi_{j=1}^np_j\Ac\to
\ug(\Ac)p_n\times_{D^U(\delta)p_n}T\Pi_{j=1}^np_j\Ac
$$
given by 
$$
\widetilde\Phi_\delta([(ap_n,(p_jx_j)_{j=1}^n,(p_jy_j)_{j=1}^n)])
=[(\Phi_\delta(a)p_n, (p_jx_j)_{j=1}^n,(p_jy_j)_{j=1}^n)] 
$$
for every $a\in\ug(\Ac)$ and $x_j,y_j\in \Ac$.

From the above, one can obtain the other notions equivalent to the connection $\widetilde\Phi_\delta$ on the tautological bundle on $\Fl(\delta)$, such as the connector map, and also the covariant derivative associated to it.

\begin{remark}\label{CPRcon}
\normalfont
Differential properties of sets of idempotents have been studied in several papers,  
for instance in the classical reference  
\cite{CPR90}, where in particular a principal connection is constructed on a certain principal bundle. In this remark we point out the relationship between such a connection and the principal connections discussed earlier in the present section.

For any integer $n\ge1$ let us define 
$$
Q_n:=\{(q_1,\dots,q_n)\in\Ac^n\mid q_jq_k=\delta_{jk}q_j\text{ for }1\le j,k\le n\text{ and }q_1+\cdots+q_n=\1\}
$$
where $\delta_{jk}=1$ if $j=k$ and $\delta_{jk}=0$ otherwise. 
The set $Q_n$ is a closed submanifold of $\Ac^n$ by \cite[Th. 1.7]{CPR90}. 
The Banach-Lie group $\Ac^\times$ acts on that manifold by 
$$
\Ac^\times\times Q_n\to Q_n,\quad (g,(q_1,\dots,q_n))\mapsto (gq_1g^{-1},\dots,gq_ng^{-1})
$$
and for every $q=(q_1,\dots,q_n)\in Q_n$ we denote its orbit by~$V_q$. 
It follows by \cite[Th. 2.1]{CPR90} that every orbit $V_q$ is a closed and open subset of~$Q_n$, 
and the orbit map of the above action, $\pi_q\colon\Ac^\times\to V_q$, $g\mapsto (gq_1g^{-1},\dots,gq_ng^{-1})$, 
is a principal bundle with its structural group $(\Ac^\times)_q:=\{g\in\Ac^\times\mid gq_j=q_jg\text{ for }j=1,\dots,n\}$. 
A natural principal connection on that principal bundle was constructed in \cite[Sect. 4]{CPR90}, 
its corresponding connection 1-form at $\1\in\Ac^\times$ being given by the idempotent linear map
$$
\theta\colon\Ac\to\Ac,\quad \theta(a)=\sum_{j=1}^nq_jaq_j
$$
(see particularly \cite[Rem. 4.11]{CPR90}).

There exists for every $n\ge0$ a bijection $\alpha$ between $Q_{n+1}$ and $\Pc_n(\Ac)$ given by 
$\alpha(q)_k:=\sum_{j=1}^k q_j\in\Pc_n(\Ac)$, $k=1,\dots, n$, for every $q=(q_1,\dots,q_{n+1})\in Q_{n+1}$ 
with inverse $\alpha^{-1}(\delta)_j:=p_j-p_{j-1}\in Q_{n+1}$, $j=1,\dots,n+1$ (with $p_{n+1}=\1$), for every 
$\delta=(p_1,\dots,p_n)$. Since the mappings $\alpha$ and its inverse $\alpha^{-1}$ are smooth (holomorphic, indeed) 
on $\Ac^{n+1}$ and $\Ac^n$ respectively, one has that $\Pc_n(\Ac)$ is also a closed submanifold of $\Ac^n$, diffeomorphic to $Q_{n+1}$. The diffeomorphism $\alpha$ induces by restriction the corresponding diffeomorphism between $V_q$ and $O(\delta)$ where 
$O(\delta):=\{(gp_1g^{-1},\dots,gp_ng^{-1})\in\Ac^n\mid g\in\Ac^{-1}\}$, $q=(q_1,\dots,q_{n+1})\in Q_{n+1}$ and 
$\delta=(p_1,\dots,p_n):=\alpha(q)$.
In this way, the mapping $\pi_\delta\colon\Ac^\times\to O(\delta)$, $g\mapsto (gp_1g^{-1},\dots,gp_ng^{-1})$ is a principal bundle with structure group 
$D(\delta)=(\Ac^\times)_q$ since for any $g\in\Ac^\times$ one has $gq_j=q_jg$ ($j=1,\dots,n+1$) if and only if $gp_k=p_kg$ 
($k=1,\dots,n$). Obviously, the bundles $\pi_\delta$ and $\pi_q$ are isomorphic, and we have that  the diagonal truncation $\Phi_\delta\colon\Ac\to D(\delta)$ defined in subsection 3.1 is the connection 1-form at $\1\in\Ac^\times$, for the principal bundle $\Ac^\times\to\Pc_n(\Ac)$, which corresponds to $\theta$ under the above isomorphism.

A significant difference between the paper \cite{CPR90} and the present paper (apart from the introduction here of the Stiefel bundle on flags) is that we deal with generalized Grassmannians on $\Ac^n$, that is, with equivalence {\it classes} of $n$-uples of idempotents rather than $n$-uples of idempotents. 
Thus for the principal bundle $A^\times\to\Fl([\delta])$ 
the structure group is $\Delta(\delta)^\times$ instead $D(\delta)^\times$ and $\Phi_\delta$ is not a connection in this case. 
However, in the unitary case $U(\Ac)\cap\Delta(\delta)^\times=U(\Ac)\cap D(\delta)^\times$, as we have seen before, 
and then the map $\Phi_\delta\vert_{\ug(\Ac)}$ discussed in the paragraph after Remark~\ref{pre-reduction} defines a principal connection on the principal bundle $U(\Ac)\to\Fl(\delta)$, when $\delta=(p_1,\dots,p_n)$ such that 
$p_j=p_j^*$ ($j=1,\dots,n$).

More precisely,
fix $q^0=(q^0_1,\dots,q^0_{n+1})\in Q_n$ with $q^0_j=(q^0_j)^*$ for $j=1,\dots,n+1$, 
and take  
$\delta^0:=\alpha(q^0)$. Put 
$V_{q^0,\perp}
:=V_{q^0}\cap\{(q_1,\dots,q_{n+1})\in Q_{n+1}\mid q_j=q_j^*\ \rm{ for }\ j=1,\dots,n+1\}$. 
One has the following commutative diagram 
\begin{equation*}
\begin{CD}
U(\Ac) @>{\1_{U(\Ac)}}>> U(\Ac)@>>> \Ac^\times=\Ac^\times \\
@VV{}V @VV{}V @VV{\pi_{q^0}\simeq\pi_{\delta^0}}V \\
\Fl(\delta^0)  @>{\alpha^{-1}}>> V_{q^0,\perp} @>{\iota}>> V_{q^0}\simeq O(\delta^0) 
\end{CD}
\end{equation*} 
where 
$\iota$ is the inclusion map, while the upper horizontal arrow from the above diagram is 
the inclusion map $U(\Ac) \hookrightarrow \Ac^\times$, 
and the left vertical arrow in that diagram is the principal bundle 
$u\mapsto (up_1u^{-1},\dots,up_{n-1}u^{-1})$ with structural group 
$\Delta(\delta)\cap U(\Ac)$ 
as in equation~\eqref{unitaryBundle}. 
It is easily seen that the map $\iota\circ\alpha^{-1}$ is an embedding of $\Fl(\delta^0)$ as a submanifold of $V_{q^0}$, 
and via the identification $\Fl(\delta^0)\simeq(\iota\circ\alpha^{-1})(\Fl(\delta^0))$.

The above commutative diagram gives a morphism of principal bundles as in the paragraph prior to \cite[Ch. I, Prop. 5.3]{KN63}, 
and the principal bundle $U(\Ac)\to \Fl(\delta^0)$ is a reduced bundle of the 
restriction of the principal bundle $\pi_{q^0}$ to the submanifold 
$\Fl(\delta^0)\simeq(\iota\circ\alpha^{-1})(\Fl(\delta^0))\hookrightarrow V_{q^0}$. 
Equivalently, $U(\Ac)\to \Fl(\delta^0)$ is a reduced bundle of the 
pullback principal bundle $(\iota\circ\alpha^{-1})^*(\pi_{q^0})$. 
In this context, the reduction of the principal connection on $\iota^*(\pi_{q^0})$ defined by the above idempotent map~$\theta$ 
is the connection defined by the map $\Phi_\delta\vert_{\ug(\Ac)}$ discussed in the paragraph after Remark~\ref{pre-reduction}. 
\end{remark}

\begin{remark}
\normalfont
The structure on the principal bundle 
$\Vc(\delta,\Ac)\to\Fl_\Ac([\delta])$ associated to the projection map $E_\delta\colon \Ac\to\Delta(\delta)$ and its restriction
$\widetilde E_\delta\colon\ug(\Ac)p_n\to\Delta(\delta)p_n$ (see \eqref{connectionE} and above subsection~\ref{subsect3.3}),  
show up simultaneously with the reductive structure on the unitary principal bundle $\Vc^U(\delta,\Ac)\to\Fl(\delta)$ associated with the principal connection 
$\Phi_\delta^U\colon\ug(\Ac)p_n\to D^{\ug}(\delta)p_n$ (see Proposition 
\ref{principalUniStiefel}). 
The simultaneous occurrence of the aforementioned structures suggests to investigate the relationship between almost holomorphic structures on the Stiefel bundle and horizontal distributions (connections) on the unitary Stiefel bundle. This topic will be considered in a forthcoming paper. 
\end{remark}

\begin{remark}
\normalfont
To complete the present discussion, let us point out some facts about frame bundles in the present setting.

For $\delta=(p_1,\dots,p_n)$ as before, and $k=1,\dots,n$ one has the $\Ac^\times$-equivariant submersion 
$$
\pr_k\colon\Fl(\delta)\to\Gr(p_k,\Ac), 
\quad [(gp_1g^{-1},\dots,gp_ng^{-1})]\mapsto [gp_kg^{-1}],
$$
which is just the $k$-th component of the natural $\Ac^\times$-equivariant inclusion 
$$
\Fl(\delta)\hookrightarrow \Gr(p_1,\Ac)\times\cdots\times\Gr(p_n,\Ac).
$$

Let $\sigma_k\colon\Vc(p_k,\Ac)\to\Gr(p_k,\Ac)$ denote the Stiefel bundle
with its unitary reduction $\sigma_k^U\colon\Vc^U(p_k,\Ac)\to\Gr(p_k,\Ac)$.
The pullback bundle 
$$
\pr_k^*(\sigma_k^U)\colon \pr_k^*(\Vc^U(p_k,\Ac))\to\Fl(\delta),
$$
is a principal bundle whose structure group is $U(\widehat{p_k}\Ac\widehat{p_k})$ just like the structure group of $\sigma_k^U$. 
The principal bundle can be regarded as the \emph{$k$-th partial frame bundle on the flag manifold $\Fl(\delta)$}, 
because it consists of frames that take into account the mutual positions only of the terms $p_1,\dots,p_k$ of the flag $\delta$. 

It is easily seen that the bundle $\sigma_k^U$ is $U(\Ac)$-equivariantly isomorphic to the natural principal bundle 
$$U(\Ac)/U(\widehat{p_k}\Ac\widehat{p_k})\to U(\Ac)/(U(p_k\Ac p_k)\times U(\widehat{p_k}\Ac\widehat{p_k}))$$
defined by the canonical embedding $U(\widehat{p_k}\Ac\widehat{p_k})\hookrightarrow U(p_k\Ac p_k)\times U(\widehat{p_k}\Ac\widehat{p_k})$. 
Hence $\sigma_k^U$ has a canonical principal connection $\Phi_k$ 
that is both left $U(\Ac)$-invariant and right $U(\widehat{p_k}\Ac\widehat{p_k})$-invariant, 
just as in the finite-dimensional situation studied in \cite[Prop. 1]{NR61} and \cite{NR63}.  
Furthermore, this principal connection induces linear connections and their corresponding covariant derivatives on all vector bundles associated with $\sigma_k^U$ by the construction of subsection~\ref{subs2.4} (see also \cite[Th. 2.2]{BG14}). 
These associated vector bundles can be constructed from $*$-representations of $\Ac$. 
And finally, the vector bundles associated with $\sigma_k^U$ can be pulled back to $\Fl(\delta)$ via $\pr_k$, 
just as in the above construction of partial frame bundles, and the linear connections are transferred by that pullback operation 
using Proposition~\ref{C9}. 

In the special case when $\Ac$ is a $W^*$-algebra, Hermitian vector bundles on flag manifolds 
constructed in the above way occur in certain problems of noncommutative spectral theory, 
as shown in \cite[Sect. 3]{Be02}. 
Therefore we expect that the differential geometric methods developed in the present paper 
and in particular the infinitesimal symmetries of Grassmann and flag manifolds associated to Banach algebras 
might have applications in the spectral theory of operator Lie algebras, 
besides the applications to reproducing kernels on vector bundles that are suggested by \cite{BG14}, \cite{BG15a}, and \cite{BG15b}. 
\end{remark}

\subsection*{Acknowledgment} 
We wish to thank the Referee for carefully reading our manuscript and for 
the suggestion that we should clarify the relation between 
the principal connection from Proposition~\ref{principalUniStiefel} and the principal connection 
constructed in the classical reference \cite[Section 4]{CPR90}. See Remark \ref{CPRcon}.


\begin{thebibliography}{nunitary}

\bibitem[ACS01]{ACS01}
E.~Andruchow, G.~Corach, D.~Stojanoff, 
Projective space of a $C^*$-module. 
\textit{Infin. Dimens. Anal. Quantum Probab. Relat. Top.} 
\textbf{4} (2001), no.~3, 289--307.


\bibitem[Be02]{Be02}
D.~Belti\c t\u a, 
Spectra for solvable Lie algebras of bundle endomorphisms. 
\textit{Math. Ann.} \textbf{324} (2002), no.~2, 405--429.



\bibitem[BG09]{BG09}
D.~Belti\c t\u a, J.E.~Gal\'e,
On complex infinite-dimensional Grassmann manifolds. 
\textit{Complex Anal. Oper. Theory} \textbf{3} (2009), no.~4, 739-758.


\bibitem[BG14]{BG14}
D.~Belti\c t\u a, J.E.~Gal\'e,
Linear connections for reproducing kernels on vector bundles.
\textit{Math. Z.} \textbf{277} (2014), no. 1--2, 29--62.

\bibitem[BG15]{BG15a}
D.~Belti\c t\u a, J.E.~Gal\'e,
\textit{Geometric perspectives of reproducing kernels}. In: Operator Theory 
(Reference work), D. Alpay (ed.), Springer 2015, ISBN: 978-3-0348-0692-3.

\bibitem[BG16]{BG15b}
D.~Belti\c t\u a, J.E.~Gal\'e,
Reproducing kernels and positivity of vector bundles in infinite dimensions. 
In: M. de Jeu, B. de Pagter, O. van Gaans, and M. Veraar (eds.), 
\textit{Ordered Structures and Applications}, 
Trends in Math., Birkh\"auser, Basel, 2016, pp.~51--75. 


\bibitem[BN10]{BN10}
D.~Belti\c t\u a, K.-H.~Neeb, 
Geometric characterization of Hermitian algebras with continuous inversion. 
\textit{Bull. Aust. Math. Soc.} \textbf{81} (2010), no.~1, 96--113. 

\bibitem[Bo67]{Bo67}
N.~Bourbaki,
\textit{\'El\'ements de Math\'ematique. Fasc. XXXIII. 
Vari\'et\'es diff\'erentielles et analytiques}.
   Fascicule de r\'esultats (Paragraphes 1 \`a 7). 
Actualit\'es Scient. et Industr., No.~1333, Hermann, Paris, 1967.

\bibitem[CPR90]{CPR90}
G. Corach, H. Porta, L. Recht, 
Differential geometry of systems of projections in Banach algebras. 
\textit{Pacific J. Math.} \textbf{143} (1990), no. 2, 209--228.

\bibitem[CD78]{CD78}
M.J. Cowen, R.G. Douglas, 
Complex geometry and operator theory. 
\textit{Acta Math.} \textbf{141} (1978), no. 3--4, 187--261.

\bibitem[CG99]{CG99} 
G.~Corach, J.E.~Gal\'e, 
On amenability and geometry of spaces of bounded representations, 
\textit{J. London Math. Soc.} {\bf 59} (1999), 
no.~2, 311--329. 

\bibitem[DG01]{DG01}
M.J.~Dupr\'e, J.F.~Glazebrook, 
The Stiefel bundle of a Banach algebra, 
\textit{Integral Equations Operator Theory} 
\textbf{41} (2001), no.~3, 264--287. 

\bibitem[DG02]{DG02}
M.J.~Dupr\'e, J.F.~Glazebrook, 
Holomorphic framings for projections in a Banach algebra, 
\textit{Georgian Math. J.} 
\textbf{9} (2002), no.~3, 481--494.

\bibitem[El67]{El67}
H.I.~El\u\i asson, 
Geometry of manifolds of maps, 
\textit{J. Diff. Geometry} \textbf{1} (1967), 169--194. 

\bibitem[Es83]{Es83}
J. Esterle,
Polynomial connections between projections in Banach algebras,
\textit{Bull. London Math. Soc.}
\textbf{15} (1983), 253--254. 

\bibitem[EG04]{EG04}
J. Esterle and J. Giol,
Polynomial and polygonal connections between idempotents in finite-dimensional real algebras, 
\textit{Bull. London Math. Soc.} 
\textbf{36} (2004), 378--382. 

\bibitem[Gi03]{Gi03}
J. Giol,
Arcs d'idempotents dans les alg\`ebres de Banach, 
\textit{Th\`ese de doctorat, Univ. Bordeaux I}.

\bibitem[GH78]{GH78}
Ph. Griffiths and J. Harris, 
\textit{Principles of algebraic geometry}. 
Pure and Applied Mathematics. Wiley-Interscience, New York, 1978.


\bibitem[KN63]{KN63} 
S.~Kobayashi, K.~Nomizu, 
\textit{Foundations of differential geometry}. Vol. I. 
Interscience Publishers, John Wiley \&\ Sons, New York-London, 1963. 

\bibitem[KN69]{KN69} 
S.~Kobayashi, K.~Nomizu, 
\textit{Foundations of differential geometry}. Vol. II. 
Interscience Publishers, John Wiley \&\ Sons, Inc., New York-London-Sydney, 1969. 


\bibitem[KM97]{KM97a}
A.~Kriegl, P.W.~Michor, 
\textit{The Convenient Setting of Global Analysis}. 
 Mathematical Surveys and Monographs, 53. 
 American Mathematical Society, Providence, RI, 1997. 

\bibitem[La01]{La01}
S.~Lang,
{\it Fundamentals of Differential Geometry} (corrected second printing),
Graduate Texts in Mathematics, 191. Springer-Verlag,
New-York,
2001.

\bibitem[Lr11]{Lr11}
G.~Larotonda, 
\textit{Estructuras geom\'etricas para las variedades de Banach}. 
Colecci\'on Ciencia Innovaci\'on y Desarrollo, Univ. Nacional de General Sarmiento, 
Buenos Aires, 2011. 

\bibitem[Le68]{Le68}
D.~Lehmann, 
Quelques propri\'et\'es des connexions induites,
\textit{Bull. Soc. Math. France Suppl. M\'em.} 
\textbf{16} (1968), 7--99.

\bibitem[Ma04]{Ma04}
J.-P.~Magnot, 
Structure groups and holonomy in infinite dimensions, 
\textit{Bull. Sci. Math.} \textbf{128} (2004), no.~6, 513--529.

\bibitem[MS97]{MS97} 
M.~Martin, N.~Salinas, 
Flag manifolds and the Cowen-Douglas theory. 
\textit{J. Operator Theory} \textbf{38} (1997), no.~2, 329--365.

\bibitem[MR92]{MR92}
L.E.~Mata-Lorenzo, L.~Recht, 
Infinite-dimensional homogeneous reductive spaces,  
{\it Acta Cient. Venezolana}  
{\bf 43}(1992),  no.~2, 76--90.

\bibitem[NR61]{NR61}
M.S.~Narasimhan, S.~Ramanan, 
Existence of universal connections, 
\textit{Amer. J. Math.} \textbf{83} (1961), 563--572.

\bibitem[NR63]{NR63}
M.S.~Narasimhan, S.~Ramanan, 
Existence of universal connections. II, 
\textit{Amer. J. Math.} \textbf{85} (1963), 223--231.

\bibitem[Ne00]{Ne00}
K.-H.~Neeb,
\textit{Holomorphy and Convexity in Lie Theory}, 
 de Gruyter Expositions in Mathematics 28, 
Walter de Gruyter \& Co., Berlin, 2000.

\bibitem[Pe69]{Pe69}
J.-P.~Penot, 
Connexion lin\'eaire d\'eduite d'une famille de connexions lin\'eaires 
par un foncteur vectoriel multilin\'eaire, 
\textit{C. R. Acad. Sci. Paris S\'er. A-B} \textbf{268} (1969), A100--A103.

\bibitem[Pe70]{Pe70}
J.-P.~Penot, 
Sur le th\'eor\`eme de Frobenius. 
\textit{Bull. Soc. Math. France} \textbf{98} (1970), 47--80. 

\bibitem[PR86]{PR86}
H.~Porta, L.~Recht, 
Classification of linear connections, 
\textit{J. Math. Anal. Appl.} \textbf{118} (1986), no.~2, 547--560.

\bibitem[PR87]{PR87}
H.~Porta, L.~Recht, 
Spaces of projections in a Banach algebra,  
\textit{Acta Cient. Venezolana} \textbf{38} (1987), no.~4, 408--426.

\bibitem[Sch80]{Sch80}
R.~Schlafly,  
Universal connections,  
\textit{Invent. Math.} \textbf{59} (1980), no.~1, 59--65.

\bibitem[Sch82]{Sch82}
R.~Schlafly, 
Universal connections: the local problem, 
\textit{Pacific J. Math.} \textbf{98} (1982), no.~1, 157--171. 


\bibitem[Tr85]{Tr85}
M. Tr\'emon,
Polyn\^omes de degr\'e minimum connectant deux projections dans une alg\`ebre de Banach,
\textit{Linear Algebra Appl.}
\textbf{64} (1985), 115--132.

\bibitem[Up85]{Up85}
H.~Upmeier,
{\it Symmetric Banach Manifolds and Jordan $C^*$-algebras},
North-Holland Mathematics Studies, 104. 
Notas de Matem\`atica, 96. North-Holland Publishing Co.,
Amsterdam,
1985.

\bibitem[Va74]{Va74}
E.E.~Vassiliou, 
Christoffel symbols and connection forms on 
infinite-dimensional fibre bundles, 
\textit{Bull. Soc. Math. Gr\`ece (N.S.)} \textbf{15} (1974), 115--122.



\end{thebibliography}
\end{document}